\documentclass[reqno,12pt, centertags]{amsart}
\usepackage{amsmath,amsthm,amscd,amssymb}
\usepackage{upref}
\usepackage{latexsym}
\usepackage{lscape}
\usepackage{graphicx}
\usepackage{color}
\usepackage{calligra}

\DeclareMathAlphabet{\mathcalligra}{T1}{calligra}{m}{n}
\DeclareFontShape{T1}{calligra}{m}{n}{<->s*[2.2]callig15}{}

\makeatletter
\def\theequation{\@arabic\c@equation}

\newcommand{\mas}{\operatorname{Mas}}

\newcommand{\Mas}{\operatorname{Mas}}
\newcommand{\Mor}{\operatorname{Mor}}

\newcommand{\beq}{\begin{equation}}
\newcommand{\enq}{\end{equation}}

\newcommand{\bbC}{{\mathbb{C}}}

\newcommand{\bi}{\bibitem}

\renewcommand{\ln}{\text{\rm ln}}

\newcommand{\scripty}[1]{\ensuremath{\mathcalligra{#1}}}

\numberwithin{equation}{section}

\newcommand{\dom}{\operatorname{dom}}

\newcommand{\sech}{\operatorname{sech}}

\renewcommand{\ker}{\operatorname{ker}}

\theoremstyle{plain}
\newtheorem{theorem}{Theorem}[section]

\newtheorem{lemma}[theorem]{Lemma}

\theoremstyle{definition}
\newtheorem{definition}[theorem]{Definition}

\newtheorem{remark}[theorem]{Remark}

\newtheorem{claim}[theorem]{Claim}

\allowdisplaybreaks

\title{The Maslov and Morse indices for Schr\"odinger operators on $\mathbb{R}$}

\author[P.\ Howard, Y.\ Latushkin, and A.\ Sukhtayev]{P.\ Howard, Y. Latushkin, and A.\ Sukhtayev}
\address{Mathematics Department,
Texas A\&M University, College Station, TX 77843, USA}
\address{Mathematics Department,
University of Missouri, Columbia, MO 65211, USA}
\address{Mathematics Department,
Indiana University, Bloomington, IN 47405, USA}

\email{phoward@math.tamu.edu}
\email{latushkiny@missouri.edu}
\email{alimsukh@iu.edu}
\date{\today}

\begin{document}

\begin{abstract} 
Assuming a symmetric potential that approaches constant endstates
with a sufficient asymptotic rate, we relate the Maslov and Morse
indices for Schr\"odinger operators on $\mathbb{R}$. In particular,
we show that with our choice of convention, the Morse index is 
precisely the negative of the Maslov index. 
\end{abstract}

\maketitle

\section{Introduction} \label{introduction}

We consider eigenvalue problems 
\begin{equation} \label{main}
\begin{aligned}
Hy &:= -y'' + V(x) y = \lambda y, \\
\dom (H) &= H^1 (\mathbb{R}),
\end{aligned}
\end{equation}
and also (for any $s \in \mathbb{R}$)
\begin{equation} \label{main_s}
\begin{aligned}
H_s y &:= -y'' + sy' +  V(x) y = \lambda y, \\
\dom (H_s) &= H^1 (\mathbb{R}),
\end{aligned}
\end{equation}
where $y \in \mathbb{R}^n$ and $V \in C(\mathbb{R})$ is a real-valued 
symmetric matrix satisfying the following asymptotic conditions:

\medskip
\noindent
{\bf (A1)} The limits $\lim_{x \to \pm \infty} V(x) = V_{\pm}$ exist, and 
for all $M \in \mathbb{R}$, 
\begin{equation*}
\int_{-M}^{\infty} (1 + |x|) |V(x) - V_+| dx < \infty; 
\quad
\int_{-\infty}^M (1+|x|) |V(x) - V_-| dx < \infty.
\end{equation*} 

\medskip
\noindent
{\bf (A2)} The eigenvalues of $V_{\pm}$ are all non-negative. We denote 
the smallest among all these eigenvalues $\nu_{\min} \ge 0$.

\medskip

Our particular interest lies in counting the number of negative
eigenvalues for $H$ (i.e., the Morse index). We proceed by relating 
the Morse index to the Maslov index, which is described in 
Section \ref{maslov_section}. In essence, we find that the 
Morse index can be computed in terms of the Maslov index, and that
while the Maslov index is less elementary than the Morse index, 
it can be computed (numerically) in a relatively straightforward way. 

The Maslov index has its origins in the work of V. P. Maslov 
\cite{Maslov1965a} and subsequent development by V. I. Arnol'd
\cite{arnold67}. It has now been studied extensively, both 
as a fundamental geometric quantity \cite{BF98, CLM, F, P96, rs93}
and as a tool for counting the number of eigenvalues on specified
intervals \cite{BJ1995, BM2013, CDB09, CDB11, Chardard2009, 
CJLS2014, DJ11, FJN03, J88, J88a, JM2012}. In this latter context, 
there has been a strong resurgence of interest following the 
analysis by Deng and Jones (i.e., \cite{DJ11}) for multidimensional
domains. Our aim in the current analysis is to rigorously develop
a relationship between the Maslov index and the Morse index in 
the relatively simple setting of (\ref{main}). Our approach is 
adapted from \cite{CJLS2014, DJ11, HS}. 

As a starting point, we define what we will mean by a {\it Lagrangian
subspace} of $\mathbb{R}^{2n}$.

\begin{definition} \label{lagrangian_subspace}
We say $\ell \subset \mathbb{R}^{2n}$ is a Lagrangian subspace
if $\ell$ has dimension $n$ and
\begin{equation*}
(Jx, y)_{\mathbb{R}^{2n}} = 0, 
\end{equation*} 
for all $x, y \in \ell$. Here, $(\cdot, \cdot)_{\mathbb{R}^{2n}}$ denotes
Euclidean inner product on $\mathbb{R}^{2n}$, and  
\begin{equation*}
J = 
\begin{pmatrix}
0 & -I_n \\
I_n & 0
\end{pmatrix},
\end{equation*}
with $I_n$ the $n \times n$ identity matrix. We sometimes adopt standard
notation for symplectic forms, $\omega (x,y) = (Jx, y)_{\mathbb{R}^{2n}}$.
In addition, we denote by $\Lambda (n)$ the collection of all Lagrangian 
subspaces of $\mathbb{R}^{2n}$, and we will refer to this as the 
{\it Lagrangian Grassmannian}.
\end{definition}

A simple example, important for intuition, is the case $n = 1$, for which 
$(Jx, y)_{\mathbb{R}^{2}} = 0$ if and only if $x$ and $y$ are linearly 
dependent. In this case, we see that any line through the origin is a 
Lagrangian subspace of $\mathbb{R}^2$. More generally, any Lagrangian 
subspace of $\mathbb{R}^{2n}$ can be spanned by a choice of $n$ linearly 
independent vectors in $\mathbb{R}^{2n}$. We will find it convenient to collect
these $n$ vectors as the columns of a $2n \times n$ matrix $\mathbf{X}$, 
which we will refer to as a {\it frame} (sometimes {\it Lagrangian frame}) 
for $\ell$. Moreover, we will often write $\mathbf{X} = {X \choose Y}$, 
where $X$ and $Y$ are $n \times n$ matrices.

Suppose $\ell_1 (\cdot), \ell_2 (\cdot)$ denote paths of Lagrangian 
subspaces $\ell_i: I \to \Lambda (n)$, for some parameter interval 
$I$. The Maslov index associated with these paths, which we will 
denote $\Mas (\ell_1, \ell_2; I)$, is a count of the number of times
the paths $\ell_1 (\cdot)$ and $\ell_2 (\cdot)$ intersect, counted
with both multiplicity and direction. (Precise definitions of what we 
mean in this context by {\it multiplicity} and {\it direction} will be
given in Section \ref{maslov_section}.) In some cases, the Lagrangian 
subspaces will be defined along some path 
\begin{equation*}
\Gamma = \{ (x (t), y(t)): t \in I\},
\end{equation*}    
and when it's convenient we'll use the notation 
$\Mas (\ell_1, \ell_2; \Gamma)$. 

We will verify in Section \ref{ODEsection} that under our assumptions
on $V (x)$, and for $\lambda < \nu_{\min}$, (\ref{main}) will have 
$n$ linearly independent solutions 
that decay as $x \to -\infty$ and $n$ linearly independent solutions
that decay as $x \to +\infty$. We express these respectively as 
\begin{equation*}
\begin{aligned}
\phi_{n+j}^- (x; \lambda) &= e^{\mu_{n+j}^- (\lambda) x} (r_j^- + \mathcal{E}_j^- (x;\lambda)) \\
\phi_j^+ (x; \lambda) &= e^{\mu_j^+ (\lambda) x} (r_{n+1-j}^+ + \mathcal{E}_j^+ (x;\lambda)),
\end{aligned}
\end{equation*}
with also
\begin{equation*}
\begin{aligned}
\partial_x \phi_{n+j}^- (x; \lambda) &= e^{\mu_{n+j}^- (\lambda) x} (\mu_{n+j}^- r_j^- + \tilde{\mathcal{E}}_j^- (x;\lambda)) \\
\partial_x \phi_j^+ (x; \lambda) &= e^{\mu_j^+ (\lambda) x} (\mu_j^+ r_{n+1-j}^+ + \tilde{\mathcal{E}}_j^+ (x;\lambda)),
\end{aligned}
\end{equation*}
for $j = 1,2,\dots,n$, where the nature of the $\mu_j^{\pm}$, $r_j^{\pm}$, and 
$\mathcal{E}_j^{\pm} (x; \lambda), \tilde{\mathcal{E}}_j^{\pm} (x; \lambda)$ 
are developed in Section \ref{ODEsection}. The only details we'll need for
this preliminary discussion are: (1) that we can continuously extend these
functions to $\lambda = \nu_{\min}$ (though they may no longer decay at 
one or both endstates), and (2) that under assumptions (A1) and (A2)
\begin{equation} \label{limits}
\lim_{x \to \pm \infty} \mathcal{E}_j^{\pm} (x;\lambda) = 0; 
\quad 
\lim_{x \to \pm \infty} \tilde{\mathcal{E}}_j^{\pm} (x;\lambda) = 0.
\end{equation}

We will verify in Section \ref{ODEsection} that if we create a frame
$\mathbf{X}^- (x; \lambda) = {X^- (x; \lambda) \choose Y^- (x; \lambda)}$ 
by taking $\{\phi_{n+j}^-\}_{j=1}^n$ as the columns of $X^-$ and 
$\{\phi_{n+j}^{- \, '}\}_{j=1}^n$ as the respective columns of $Y^-$ then 
$\mathbf{X}^-$ is a frame for a Lagrangian subspace, which we will 
denote $\ell^-$. Likewise, we can create a frame
$\mathbf{X}^+ (x; \lambda) = {X^+ (x; \lambda) \choose Y^+ (x; \lambda)}$ 
by taking $\{\phi_j^+\}_{j=1}^n$ as the columns of $X^+$ and 
$\{{\phi_j^+}'\}_{j=1}^n$ as the respective columns of $Y^+$. Then 
$\mathbf{X}^+$ is a frame for a Lagrangian subspace, which we will 
denote $\ell^+$.

In constucting our Lagrangian frames, we can view the exponential 
multipliers $e^{\mu_j^{\pm} x}$ as expansion coefficients, and if 
we drop these off we retain frames for the same spaces. That is, we 
can create an alternative frame for 
$\ell^-$ by taking the expressions $r_j^- + \mathcal{E}_j^- (x;\lambda)$
as the columns of $X^-$ and the expressions 
$\mu_{n+j}^- r_j^- + \tilde{\mathcal{E}}_j^- (x;\lambda)$ as the 
corresponding columns for $Y^-$. Using (\ref{limits}) we see that 
in the limit as $x$ tends to $-\infty$ we obtain the frame 
$\mathbf{R}^- (\lambda) = {R^- \choose S^- (\lambda)}$, where 
\begin{equation*}
\begin{aligned}
R^- &= 
\begin{pmatrix}
r_1^- & r_2^- & \dots & r_n^-
\end{pmatrix} \\
S^- (\lambda) &= 
\begin{pmatrix}
\mu_{n+1}^- (\lambda) r_1^- & \mu_{n+2}^- (\lambda) r_2^- & \dots & \mu_{2n}^- (\lambda) r_n^-
\end{pmatrix}. 
\end{aligned}
\end{equation*}
(The dependence on $\lambda$ is specified here to emphasize the fact 
that $S^- (\lambda)$ depends on $\lambda$ through the multipliers
$\{\mu_{n+j}^-\}_{j=1}^n$.) We will verify in Section \ref{ODEsection}  
that $\mathbf{R}^- (\lambda)$ is the frame for a Lagrangian subspace, 
and we denote this space $\ell_{\mathbf{R}}^- (\lambda)$.

Proceeding similarly with $\ell^+$, we obtain the asymptotic Lagrangian
subspace $\ell_{\mathbf{R}}^+ (\lambda)$ with frame 
$\mathbf{R}^+ (\lambda) = {R^+ \choose S^+ (\lambda)}$,
where 
\begin{equation} \label{RplusSplus}
\begin{aligned}
R^+ &= 
\begin{pmatrix} 
r_n^+ & r_{n-1}^+ & \dots & r_1^+
\end{pmatrix} \\
S^+ (\lambda) &= 
\begin{pmatrix} 
\mu_1^+ (\lambda) r_n^+ & \mu_2^+ (\lambda) r_{n-1}^+ & \dots & \mu_n^+ (\lambda) r_1^+
\end{pmatrix}. 
\end{aligned}
\end{equation}
(The ordering of the columns of $\mathbf{R}^+$ is simply a convention, which 
follows naturally from our convention for 
indexing $\{\phi_j^+\}_{j=1}^n$.) 

Let $\bar{\Gamma}_0$ denote the contour in the $x$-$\lambda$ plane
obtained by fixing $\lambda = 0$ and letting $x$ run from $-\infty$
to $\infty$.

We are now prepared to state the main theorem of the paper.

\begin{theorem} \label{main_theorem}
Let $V \in C(\mathbb{R})$ be a real-valued symmetric matrix, and suppose (A1)
and (A2) hold. Then 
\begin{equation*}
\Mor(H) = - \Mas(\ell^-, \ell^+_{\mathbf{R}}; \bar{\Gamma}_0).
\end{equation*}
\end{theorem}

\begin{remark} The advantage of this theorem resides in the fact
that the Maslov index on the right-hand side is generally 
straightforward to compute numerically. See, for example, 
\cite{BM2013, CDB06, CDB09, CDB11, Chardard2009}, and the 
examples we discuss in Section \ref{applications_section}.
The choice of $\lambda = 0$ for $\bar{\Gamma}_0$ is not necessary
for the analysis, and indeed if we fix any $\lambda_0$ so that
$\sigma_{ess} (H) \subset [\lambda_0, \infty)$ then 
$\Mas (\ell^-, \ell_{\mathbf{R}}^+; \bar{\Gamma}_{\lambda_0})$
will be negative the count of eigenvalues of $H$ strictly
less than $\lambda_0$. Since $\bar{\Gamma}_0$ plays a distinguished
role, we refer to $\Mas (\ell^-, \ell_{\mathbf{R}}^+; \bar{\Gamma}_{0})$
as the {\it Principal Maslov Index} (following \cite{HS}).
\end{remark}

\begin{remark} In Section \ref{maslov_section} our definition of
the Maslov index will be for compact intervals $I$. We will see 
that we are able to view $\ell^- (x; 0)$ as a continuous path 
of Lagrangian subspaces on $[-1,1]$ by virtue of the change 
of variables 
\begin{equation} \label{change} 
x = \ln (\frac{1 + \tau}{1 - \tau}).
\end{equation}
\end{remark}

We will verify in Section \ref{convection_section} that for
$s \in \mathbb{R}$, any eigenvalue of $H_s$ with real part
less than or equal to $\nu_{\min}$ must be real-valued. This 
observation will allow us to construct the Lagrangian subspaces 
$\ell^-$ and $\ell^+_{\mathbf{R}}$ in that case through a 
development that looks identical to the discussion above. We 
obtain the following theorem. 

\begin{theorem} \label{main_theorem_s}
Let $V \in C(\mathbb{R})$ be a real-valued symmetric matrix, and suppose (A1)
and (A2) hold. Let $s \in \mathbb{R}$, and let $\ell^-$ and $\ell^+_{\mathbf{R}}$
denote Lagrangian subspaces developed for (\ref{main_s}). Then 
\begin{equation*}
\Mor(H_s) = - \Mas(\ell^-, \ell^+_{\mathbf{R}}; \bar{\Gamma}_0).
\end{equation*}
\end{theorem}

\begin{remark} As described in more detail in Sections 
\ref{convection_section} and \ref{applications_section},
equations of forms (\ref{main}) and (\ref{main_s})
arise naturally when a gradient system 
\begin{equation*}
u_t + F'(u) = u_{xx}
\end{equation*} 
is linearized about a stationary solution $\bar{u} (x)$
or a traveling wave solution $\bar{u} (x - st)$
(respectively). The case of solitary waves, for 
which (without loss of generality)
\begin{equation*}
\lim_{x \to \pm \infty} \bar{u} (x) = 0,
\end{equation*}
has been analyzed in 
\cite{BJ1995, CDB09, CDB11, CH2007, Chardard2009} 
(with $s \ne 0$ in \cite{BJ1995} and $s = 0$ in the
others). In particular, theorems along the lines of 
our Theorem \ref{main_theorem} (though restricted 
to the case of solitary waves) appear as Corollary 3.8
in \cite{BJ1995} and Proposition 35 in Appendix C.2 
of \cite{Chardard2009}. 
\end{remark}

{\it Plan of the paper.} In Section \ref{ODEsection} we develop 
several relatively standard results from ODE theory that will be
necessary for our construction and analysis of the Maslov index.
In Section \ref{maslov_section}, we define the Maslov index,
and discuss some of its salient properties, and in Section 
\ref{schrodinger_section} we prove Theorem \ref{main_theorem}.
In Section \ref{convection_section}, we verify that the analysis 
can be extended to the case of any $s \in \mathbb{R}$, and  finally, 
in Section \ref{applications_section} we provide some 
illustrative applications.

\section{ODE Preliminaries} \label{ODEsection}

In this section, we develop preliminary ODE results that will serve as the 
foundation of our analysis. This development is standard, and follows \cite{ZH}, pp. 
779-781. We begin by clarifying our terminology.

\begin{definition} \label{spectrum}
We define the point spectrum of $H$, denoted $\sigma_{pt} (H)$, as the set
\begin{equation*}
\sigma_{pt} (H) = \{\lambda \in \mathbb{R}: H \phi = \lambda \phi 
\, \, \text{for some} \, \, \phi \in H^1 (\mathbb{R}) \backslash \{0\}\}.
\end{equation*}
We define the essential spectrum of $H$, denoted $\sigma_{ess} (H)$, 
as the values in $\mathbb{R}$ that 
are not in the resolvent set of $H$ and are not isolated eigenvalues
of finite multiplicity.
\end{definition}

As discussed, for example, in \cite{Henry, KP}, the essential spectrum of $H$ is determined
by the asymptotic equations 
\begin{equation} \label{asymptotic}
- y'' + V_{\pm} y = \lambda y.
\end{equation}
In particular, if we look for solutions 
of the form $y (x) = e^{i k x} r$, for some scalar constant $k \in \mathbb{R}$
and (non-zero) constant vector $r \in \mathbb{R}^n$ then the essential spectrum 
will be confined to the allowable values of $\lambda$. For (\ref{asymptotic}),
we find 
\begin{equation*}
(k^2 I + V_{\pm}) r = \lambda r,
\end{equation*}
so that 
\begin{equation*}
\lambda(k) \ge \frac{(V_{\pm}r, r)}{\|r\|^2}.
\end{equation*}
Applying the min-max principle, we see that if the eigenvalues of 
$V_{\pm}$ are all non-negative then we will have 
$\sigma_{ess} (H) \subset [0, \infty)$, and more generally 
if $\nu_{\min}$ denotes the smallest eigenvalue of $V_{\pm}$ then we will
have $\sigma_{ess} (H) \subset [\nu_{\min}, \infty)$. 

Away from essential spectrum, we begin our construction of asymptotically
decaying solutions to (\ref{main}) by looking for solutions of 
(\ref{asymptotic}) of the form $\phi (x;\lambda) = e^{\mu x} r$, 
where in this case $\mu$ is a scalar function of $\lambda$, and 
$r$ is again a constant vector in $\mathbb{R}^n$. In this case, 
we obtain the relation 
\begin{equation*}
(-\mu^2 I + V_{\pm} - \lambda I) r = 0,
\end{equation*}
from which we see that the values of $\mu^2 + \lambda$
will correspond with eigenvalues of $V_{\pm}$, and 
the vectors $r$ will be eigenvectors of $V_{\pm}$. We denote the 
spectrum of $V_{\pm}$ by $\sigma (V_{\pm}) = \{\nu_j^{\pm}\}_{j=1}^n$,
ordered so that $j < k$ implies $\nu_j^{\pm} \le \nu_k^{\pm}$, and 
we order the eigenvectors correspondingly so that 
$V_{\pm} r_j^{\pm} = \nu_j^{\pm} r_j^{\pm}$ for all 
$j \in \{1, 2, \dots, n\}$. Moreover, since $V_{\pm}$ are symmetric
matrices, we can choose the set $\{r_j^{-}\}_{j=1}^n$ to be
orthonormal, and similarly for $\{r_j^{+}\}_{j=1}^n$. 

We have 
\begin{equation*}
\mu^2 + \lambda = \nu_j^{\pm} \implies 
\mu = \pm \sqrt{\nu_j^{\pm} - \lambda}.
\end{equation*}
We will denote the admissible values of $\mu$ by $\{\mu_j^{\pm}\}_{j=1}^{2n}$,
and for consistency we choose our labeling scheme so that 
$j < k$ implies $\mu_j^{\pm} \le \mu_k^{\pm}$ (for $\lambda \le \nu_{\min}$).
This leads us to the specifications 
\begin{equation*}
\begin{aligned}
\mu_j^{\pm} (\lambda) &= - \sqrt{\nu_{n+1-j}^{\pm} - \lambda} \\
\mu_{n+j}^{\pm} (\lambda) &= \sqrt{\nu_{j}^{\pm} - \lambda},
\end{aligned}
\end{equation*}
for $j = 1, 2, \dots, n$.  

We now express (\ref{main}) as a first order system, with 
$\mathbf{p} = {p \choose q} = {y \choose y'}$. We find 
\begin{equation} \label{first_order}
\frac{d \mathbf{p}}{dx} = \mathbb{A} (x; \lambda) \mathbf{p}; \quad
\mathbb{A} (x; \lambda) = 
\begin{pmatrix}
0 & I \\
V(x) - \lambda I & 0
\end{pmatrix},
\end{equation}
and we additionally set
\begin{equation*}
\mathbb{A}_{\pm} (\lambda) := \lim_{x \to \pm \infty} \mathbb{A} (x; \lambda)
= 
\begin{pmatrix}
0 & I \\
V_{\pm} - \lambda I & 0
\end{pmatrix}.
\end{equation*}
We note that the eigenvalues of $\mathbb{A}_{\pm}$ are precisely the values 
$\{\mu_j^{\pm}\}_{j=1}^{2n}$, and the associated eigenvectors are 
$\{\scripty{r}_{\,j}^{\,\pm}\}_{j=1}^n = 
\{{ r_{n+1-j}^{\pm} \choose \mu_j^{\pm} {r_{n+1-j}^{\pm}}}\}_{j=1}^n$ and 
$\{\scripty{r}_{\,n+j}^{\,\pm}\}_{j=1}^n = 
\{{r_{j}^{\pm} \choose {\mu_{n+j}^{\pm} {r_{j}^{\pm}}}}\}_{j=1}^n$.

\begin{lemma} \label{ODElemma}
Let $V \in C(\mathbb{R})$ be a real-valued symmetric matrix, and suppose (A1)
and (A2) hold. Then for any $\lambda < \nu_{\min}$ there exist $n$ 
linearly independent solutions of 
(\ref{first_order}) that decay as $x \to -\infty$ and $n$ linearly independent 
solutions of (\ref{first_order}) that decay as $x \to +\infty$. Respectively, 
we can choose these so that they can be expressed as
\begin{equation*}
\begin{aligned}
\mathbf{p}_{n+j}^- (x; \lambda) &=  e^{\mu_{n+j}^- (\lambda) x}
(\scripty{r}_{\,n+j}^{\,-} 
+ \mathbf{E}_{n+j}^-); \quad j = 1, 2, \dots, n, \\
\mathbf{p}_j^+ (x; \lambda) &=  e^{\mu_{j}^+ (\lambda) x}
(\scripty{r}_{\,j}^{\,+} 
+ \mathbf{E}_j^+); \quad j = 1, 2, \dots, n, 
\end{aligned}
\end{equation*}
where for any $\lambda_{\infty} > 0$,
$\mathbf{E}_{n+j}^{-} = \mathbf{O} ((1+|x|)^{-1})$,
uniformly for $\lambda \in [-\lambda_{\infty}, \nu_{\min}]$, and 
similarly for $\mathbf{E}_j^+$.

Moreover, there exist $n$ linearly independent solutions of 
(\ref{first_order}) that grow as $x \to -\infty$ and $n$ linearly independent 
solutions of (\ref{first_order}) that grow as $x \to +\infty$. Respectively, 
we can choose these so that they can be expressed as
\begin{equation*}
\begin{aligned}
\mathbf{p}_{j}^- (x; \lambda) &=  e^{\mu_{j}^- (\lambda) x}
(\scripty{r}_{\,j}^{\,-} 
+ \mathbf{E}_{j}^-); \quad j = 1, 2, \dots, n, \\
\mathbf{p}_{n+j}^+ (x; \lambda) &=  e^{\mu_{n+j}^+ (\lambda) x}
(\scripty{r}_{\,n+j}^{\,+} 
+ \mathbf{E}_{n+j}^+); \quad j = 1, 2, \dots, n, \\ 
\end{aligned}
\end{equation*}
where for any $\lambda_{\infty} > 0$, 
$\mathbf{E}_j^{-} = \mathbf{O} ((1+|x|)^{-1})$,
uniformly for $\lambda \in [-\lambda_{\infty}, \nu_{\min}]$,
and similarly for $\mathbf{E}_{n+j}^+$.

Finally, the solutions extend continuously as $\lambda \to \nu_{\min}$
(from the left) 
to solutions of (\ref{main}) that neither grow nor decay at the
associated endstate.  
\end{lemma}

\begin{proof}
Focusing on solutions that decay as $x \to - \infty$, we express 
(\ref{first_order}) as 
\begin{equation} \label{first_order_minus}
\frac{d \mathbf{p}}{dx} = \mathbb{A}_- (\lambda) \mathbf{p}
+ \mathcal{R}_- (x) \mathbf{p}; \quad
\mathcal{R}_- (x) = \mathbb{A} (x; \lambda) - \mathbb{A}_- (\lambda) 
=
\begin{pmatrix}
0 & 0 \\
V(x) - V_- & 0
\end{pmatrix}.
\end{equation}
We have seen that asymptotically decaying solutions to the 
asymptotic equation 
$\frac{d\mathbf{p}}{dx} = \mathbb{A}_- (\lambda) \mathbf{p}$
have the form $\mathbf{p}_{n+j}^- (x; \lambda) = e^{\mu_{n+j}^- (\lambda)} \scripty{r}_{\,n+j}^{\,-}$,
and so it's natural to look for solutions of the form 
\begin{equation*}
\mathbf{p}_{n+j}^- (x; \lambda) = e^{\mu_{n+j}^- (\lambda) x} \mathbf{z}_{n+j}^- (x;\lambda),
\end{equation*}
for which we have 
\begin{equation} \label{scaled_equation}
\frac{d \mathbf{z}_{n+j}^- (x;\lambda)}{dx} = 
(\mathbb{A}_- (\lambda) - \mu_{n+j}^- (\lambda) I) \mathbf{z}_{n+j}^- (x; \lambda)
+ \mathcal{R}_- (x) \mathbf{z}_{n+j}^- (x; \lambda). 
\end{equation}
Let $P_{n+j}^- (\lambda)$ project onto the eigenspace of $\mathbb{A}_- (\lambda)$ 
associated with eigenvalues $\sigma (\mathbb{A}_- (\lambda)) \ni \mu \le \mu_{n+j}^-$,
and let $Q_{n+j}^- (\lambda)$ likewise project onto the eigenspace of $\mathbb{A}_- (\lambda)$
associated with  $\sigma (\mathbb{A}_- (\lambda)) \ni \mu > \mu_{n+j}^-$. 
Notice particularly that there exists some $\eta > 0$ so that $\mu - \mu_{n+j}^- \ge \eta$
for all $\mu$ associated with $Q_{n+j}^- (\lambda)$. 

For some fixed $M>0$, we will look for a solution to (\ref{scaled_equation}) in 
$L^{\infty} (-\infty, -M]$ of the form 
\begin{equation} \label{contraction_map}
\begin{aligned}
\mathbf{z}_{n+j}^- (x; \lambda) &= \scripty{r}_{\,n+j}^{\,-}
+ \int_{-\infty}^x e^{(\mathbb{A}_- (\lambda) - \mu_{n+j}^- (\lambda) I) (x - \xi)}
P_{n+j}^- (\lambda) \mathcal{R}_- (\xi) \mathbf{z}_{n+j}^- (\xi; \lambda) d\xi \\
&-
\int_x^{-M} e^{(\mathbb{A}_- (\lambda) - \mu_{n+j}^- (\lambda) I) (x - \xi)}
Q_{n+j}^- (\lambda) \mathcal{R}_- (\xi) \mathbf{z}_{n+j}^- (\xi; \lambda) d\xi. 
\end{aligned}
\end{equation}
We proceed by contraction mapping, defining $\mathcal{T} \mathbf{z}_{n+j}^- (x; \lambda)$
to be the right-hand side of (\ref{contraction_map}). Let 
$\mathbf{z}_{n+j}^-, \mathbf{w}_{n+j}^- \in L^{\infty} (-\infty, -M]$, so that 
\begin{equation}
\begin{aligned}
|\mathcal{T} \mathbf{z}_{n+j}^- - \mathcal{T} \mathbf{w}_{n+j}^-|
& \le
K \|\mathbf{z}_{n+j}^- - \mathbf{w}_{n+j}^- \|_{L^{\infty} (-\infty, -M]} 
\Big{\{} \int_{-\infty}^{x} |\mathcal{R}_- (\xi)| d\xi 
+ \int_{x}^{-M} e^{\eta (x-\xi)} |\mathcal{R}_- (\xi)| d\xi
\Big{\}} \\
&=: I_1 + I_2, 
\end{aligned}
\end{equation}
for some constant $K > 0$.  

By assumption (A1) we know 
\begin{equation*}
\int_{-\infty}^0 (1+|x|) |\mathcal{R}_- (x)| dx = C < \infty,
\end{equation*}
so that
\begin{equation*}
\int_{-\infty}^{-M} (1+M) |\mathcal{R}_- (\xi)| d\xi 
\le 
\int_{-\infty}^{-M} (1+|\xi|) |\mathcal{R}_- (\xi)| d\xi
\le C, 
\end{equation*}
giving the inequality 
\begin{equation*}
\int_{-\infty}^{-M} |\mathcal{R}_- (\xi)| dx \le \frac{C}{1+M}.
\end{equation*}
Likewise, we can check that 
\begin{equation*}
\int_{x}^{-M} e^{\eta (x-\xi)} |\mathcal{R}_- (\xi)| dx 
\le \frac{C}{1+M}.
\end{equation*}
We see that 
\begin{equation*}
|\mathcal{T} \mathbf{z}_{n+j}^- - \mathcal{T} \mathbf{w}_{n+j}^-| \le 
\frac{2 K C}{1+M} \|\mathbf{z}_{n+j}^- - \mathbf{w}_{n+j}^- \|_{L^{\infty} (-\infty, -M]}, 
\end{equation*}
for all $x \in (-\infty, -M]$ so that 
\begin{equation*}
\|\mathcal{T} \mathbf{z}_{n+j}^- - \mathcal{T} \mathbf{w}_{n+j}^-\|_{L^{\infty} (-\infty, -M]} \le 
\frac{2 K C}{1+M} \|\mathbf{z}_{n+j}^- - \mathbf{w}_{n+j}^- \|_{L^{\infty} (-\infty, -M]}, 
\end{equation*}
and for $M$ large enough we have the desired contraction. Moreover, the exponential 
decay in $I_2$ allows us to see that 
\begin{equation*}
\lim_{x \to -\infty} \mathbf{z}_{n+j}^- (x; \lambda) = \scripty{r}_{\,n+j}^{\,-},
\end{equation*}
with the asymptotic rate indicated.

For continuity down to $\lambda = \mu_{\min}$, we notice that in this case some of the 
$\mu_{n+j}^-$ may be 0, so $\mathbf{p}_{n+j}^-$ will not decay as $x \to -\infty$. 
Nonetheless, our calculation remains valid, and in this case there is simply no
exponential scaling. 

Finally, we note that the case $x \to +\infty$ is similar. 
\end{proof}

Recall that we denote by $\mathbf{X}^- (x; \lambda)$ the $2n \times n$ 
matrix obtained by taking each $\mathbf{p}_{n+j}^- (x; \lambda)$ from 
Lemma \ref{ODElemma} as a column. In order to check that 
$\mathbf{X}^- (x; \lambda)$ is the frame for a Lagrangian subspace, let 
$\phi, \psi \in \{\mathbf{p}_{n+j}^- (x; \lambda)\}_{j=1}^n$, and consider 
$\omega (\phi, \psi) = (J \phi, \psi)$. First, 
\begin{equation*}
\frac{d}{dx} \omega (\phi, \psi) = 
(J \frac{d\phi}{dx}, \psi) + (J\phi, \frac{d\psi}{dx})
= (J \mathbb{A} \phi, \psi) + (J\phi, \mathbb{A} \psi).
\end{equation*}
It's important to note at this point that we can express $\mathbb{A}$
as $\mathbb{A} = J \mathbb{B}$, for the symmetric matrix 
\begin{equation*}
\mathbb{B} (x; \lambda) = 
\begin{pmatrix}
V(x) - \lambda I & 0 \\
0 & -\lambda I
\end{pmatrix}.
\end{equation*}
Consequently 
\begin{equation*}
\begin{aligned}
\frac{d}{dx} \omega (\phi, \psi) &=  
(J^2 \mathbb{B} \phi, \psi) + (J\phi, J\mathbb{B} \psi) 
= 
- (\mathbb{B} \phi, \psi) - (J^2 \phi, \mathbb{B} \psi) \\ 
&=
- (\mathbb{B} \phi, \psi) + (\phi, \mathbb{B} \psi) = 0,
\end{aligned}
\end{equation*}
where the final equality follows from the symmetry of $\mathbb{B}$.
We conclude that $\omega (\phi, \psi)$ is constant in $x$, but since 
\begin{equation*}
\lim_{x \to -\infty} \omega (\phi, \psi) = 0,
\end{equation*}
this constant must be 0. 

In order to see that this limit holds even if neither 
$\phi$ nor $\psi$ decays as $x \to -\infty$ (possible if 
$\lambda = \nu_{\min}$), we note that in this case we have 
\begin{equation*}
\begin{aligned}
\phi (x;\lambda) &= \scripty{r}_{\,n+i}^{\,-} + \mathbf{E}_{n+i}^- \\
\psi (x;\lambda) &= \scripty{r}_{\,n+j}^{\,-} + \mathbf{E}_{n+j}^-,
\end{aligned}
\end{equation*}
for some $i \ne j$ and with $\mu_{n+i}$ and $\mu_{n+j}$ both 0. Then
\begin{equation*}
\lim_{x \to -\infty} \omega (\phi, \psi) = \omega (\scripty{r}_{\,n+i}^{\,-}, \scripty{r}_{\,n+j}^{\,-})
= (J {r_i^- \choose 0}, {r_j^- \choose 0}) = 0.
\end{equation*}

Proceeding in the same way, we can verify that $\mathbf{X}^+ (x; \lambda)$ 
is also a frame for a Lagrangian subspace. 

We conclude this section by verifying that $\mathbf{R}^- (\lambda)$ 
(specified in the introduction) is the frame
for a Lagrangian subspace. To see this, we change notation a bit
from the previous calculation and take 
${\phi \choose \mu \phi}, {\psi \choose \nu \psi}
\in \{\scripty{r}_{\,n+j}^{\,-}\}_{j=1}^n$. We compute 
\begin{equation*}
\omega ({\phi \choose \mu \phi}, {\psi \choose \nu \psi})
= (J {\phi \choose \mu \phi}, {\psi \choose \nu \psi})
= (\nu - \mu) (\phi, \psi) = 0, 
\end{equation*}
where the final equality follows from orthogonality of 
the eigenvectors of $V_-$. Likewise, we find that 
$\mathbf{R}^+ (\lambda)$ is a Lagrangian subspace.

\section{The Maslov Index} \label{maslov_section}

Given any two Lagrangian subspaces $\ell_1$ and $\ell_2$, with associated 
frames $\mathbf{X}_1 = {X_1 \choose Y_1}$ and 
$\mathbf{X}_2 = {X_2 \choose Y_2}$, we can define the complex $n \times n$
matrix 
\begin{equation} \label{tildeW}
\tilde{W} = - (X_1 + i Y_1) (X_1 - i Y_1)^{-1} (X_2 - i Y_2) (X_2 + i Y_2)^{-1}.
\end{equation}  
As verified in \cite{HLS}, the matrices $(X_1 - iY_1)$ and $(X_2 + iY_2)$ are both 
invertible, and $\tilde{W}$ is unitary. We have the following theorem 
from \cite{HLS}. 

\begin{theorem} \label{intersection_theorem}
Suppose $\ell_1, \ell_2 \subset \mathbb{R}^{2n}$ are Lagrangian
subspaces, with respective frames $\mathbf{X}_1 = {X_1 \choose Y_1}$ and 
$\mathbf{X}_2 = {X_2 \choose Y_2}$, and let $\tilde{W}$ be as defined
in (\ref{tildeW}). Then 
\begin{equation*}
\dim \ker (\tilde{W} + I) = \dim (\ell_1 \cap \ell_2).
\end{equation*}   
That is, the dimension of the eigenspace of $\tilde{W}$ associated with 
the eigenvalue $-1$ is precisely the dimension of the intersection of 
the Lagrangian subspaces $\ell_1$ and $\ell_2$.
\end{theorem}

Following \cite{BF98, F}, we use Theorem \ref{intersection_theorem},
along with an approach to spectral flow introduced in \cite{P96},
to define the Maslov index. Given a parameter interval $I = [a,b]$, 
which can be normalized to $[0,1]$, we consider maps 
$\ell:I \to \Lambda (n)$, which will be 
expressed as $\ell (t)$. In order to specify a notion of continuity, 
we need to define a metric on $\Lambda (n)$, and following 
\cite{F} (p. 274), we do this in terms of orthogonal projections 
onto elements $\ell \in \Lambda (n)$. Precisely, let $\mathcal{P}_i$ 
denote the orthogonal projection matrix onto $\ell_i \in \Lambda (n)$
for $i = 1,2$. I.e., if $\mathbf{X}_i$ denotes a frame for $\ell_i$,
then $\mathcal{P}_i = \mathbf{X}_i (\mathbf{X}_i^t \mathbf{X}_i)^{-1} \mathbf{X}_i^t$.
We take our metric $d$ on $\Lambda (n)$ to be defined 
by 
\begin{equation*}
d (\ell_1, \ell_2) := \|\mathcal{P}_1 - \mathcal{P}_2 \|,
\end{equation*} 
where $\| \cdot \|$ can denote any matrix norm. We will say 
that $\ell: I \to \Lambda (n)$ is continuous provided it is 
continuous under the metric $d$. 

Given two continuous maps $\ell_1 (t), \ell_2 (t)$ on a parameter
interval $I$, we denote by $\mathcal{L}(t)$ the path 
\begin{equation*}
\mathcal{L} (t) = (\ell_1 (t), \ell_2 (t)).
\end{equation*} 
In what follows, we will define the Maslov index for the path 
$\mathcal{L} (t)$, which will be a count, including both multiplicity
and direction, of the number of times the Lagrangian paths
$\ell_1$ and $\ell_2$ intersect. In order to be clear about 
what we mean by multiplicty and direction, we observe that 
associated with any path $\mathcal{L} (t)$ we will have 
a path of unitary complex matrices as described in (\ref{tildeW}).
We have already noted that the Lagrangian subspaces $\ell_1$
and $\ell_2$ intersect at a value $t_0 \in I$ if and only 
if $\tilde{W} (t_0)$ has -1 as an eigenvalue. In the event of 
such an intersection, we define the multiplicity of the 
intersection to be the multiplicity of -1 as an eigenvalue of 
$\tilde{W}$ (since $\tilde{W}$ is unitary the algebraic and geometric
multiplicites are the same). When we talk about the direction 
of an intersection, we mean the direction the eigenvalues of 
$\tilde{W}$ are moving (as $t$ varies) along the unit circle 
$S^1$ when they cross $-1$ (we take counterclockwise as the positive direction). We note
that all of the eigenvalues certainly do not all need to be moving in 
the same direction, and that we will need to take care with 
what we mean by a crossing in the following sense: we must decide
whether to increment the Maslov index upon arrival or 
upon departure. Indeed, there are several different approaches 
to defining the Maslov index (see, for example, \cite{CLM, rs93}), 
and they often disagree on this convention. 

Following \cite{BF98, F, P96} (and in particular Definition 1.4 
from \cite{BF98}), we proceed by choosing a 
partition $a = t_0 < t_1 < \dots < t_n=b$ of $I = [a,b]$, along 
with numbers $\epsilon_j \in (0,\pi)$ so that 
$\ker\big(\tilde{W} (t) - e^{i (\pi + \epsilon_j)} I\big)=\{0\}$
for $t_{j-1} \le t \le t_j$; 
that is, $e^{i(\pi + \epsilon_j)} \in \bbC \setminus \sigma(\tilde{W} (t))$, 
for $t_{j-1} \le t \le t_j$ and $j=1,\dots,n$. 
Moreover, we notice that for each $j=1,\dots,n$ and any 
$t \in [t_{j-1},t_j]$ there are only 
finitely many values $\theta \in [0,\epsilon_j)$ 
for which $e^{i(\pi+\theta)} \in \sigma(\tilde{W} (t))$.

Fix some $j \in \{1, 2, \dots, n\}$ and consider the value
\begin{equation} \label{kdefined}
k (t,\epsilon_j) := 
\sum_{0 \leq \theta < \epsilon_j}
\dim \ker \big(\tilde{W} (t) - e^{i(\pi+\theta)}I \big).
\end{equation} 
for $t_{j-1} \leq t \leq t_j$. This is precisely the sum, along with multiplicity,
of the number of eigenvalues of $\tilde{W} (t)$ that lie on the arc 
\begin{equation*}
A_j := \{e^{i t}: t \in [\pi, \pi+\epsilon_j)\}.
\end{equation*}
(See Figure \ref{Aj}.)
The stipulation that 
$e^{i(\pi\pm\epsilon_j)} \in \bbC\setminus \sigma(\tilde{W} (t))$, for 
$t_{j-1} \le t \le t_j$
asserts that no eigenvalue can enter $A_j$ in the clockwise direction 
or exit in the counterclockwise direction during the interval $t_{j-1} \le t \le t_j$.
In this way, we see that $k(t_j, \epsilon_j) - k (t_{j-1}, \epsilon_j)$ is a 
count of the number of eigenvalues that enter $A_j$ in the counterclockwise 
direction (i.e., through $-1$) minus the number that leave in the clockwise direction
(again, through $-1$) during the interval $[t_{j-1}, t_j]$.

\begin{figure}[ht] 
\begin{center}\includegraphics[%
  width=8cm,
  height=8cm]{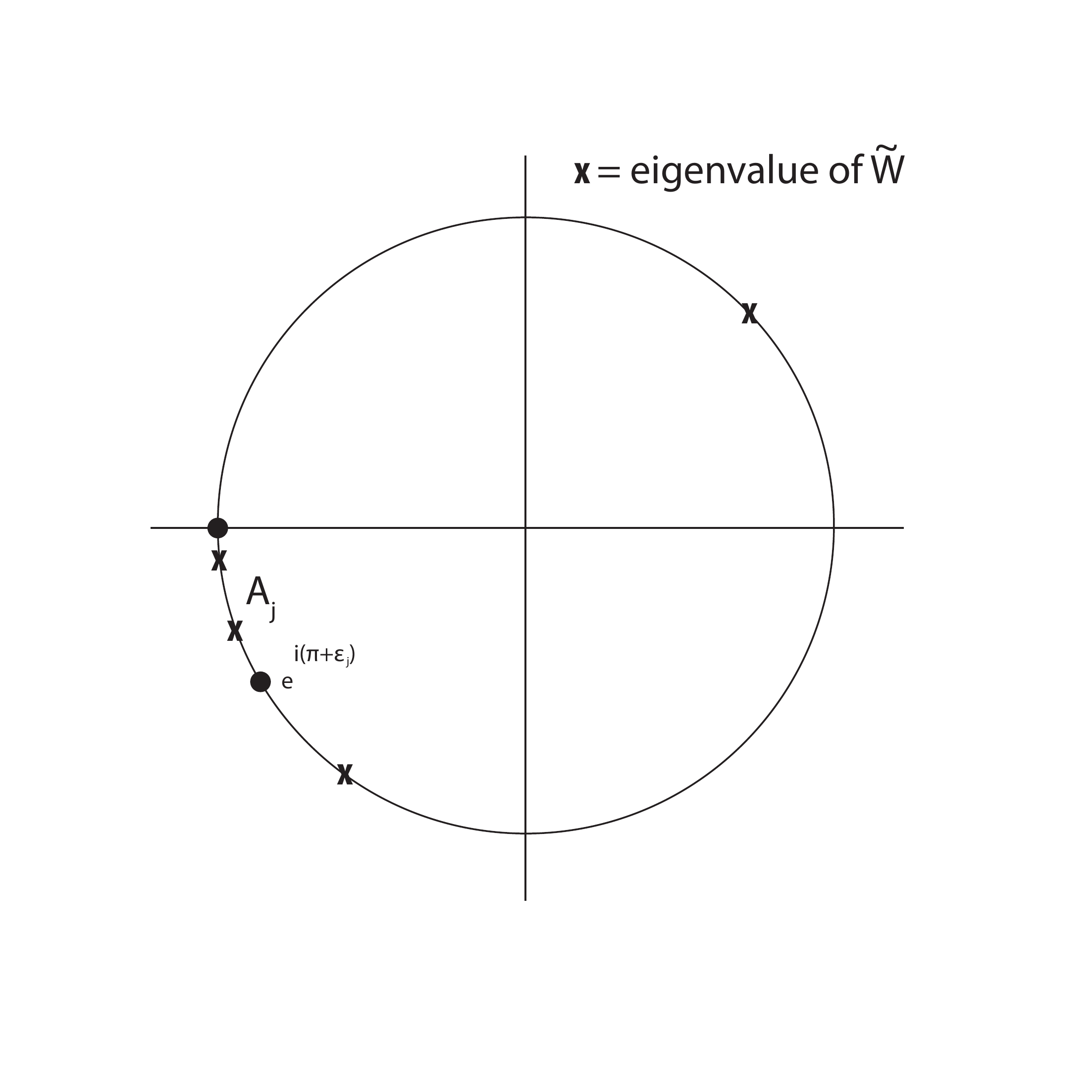}\end{center}
\caption{The arc $A_j$. \label{Aj}}
\end{figure}

In dealing with the catenation of paths, it's particularly important to 
understand this quantity if an eigenvalue resides at $-1$ at either $t = t_{j-1}$
or $t = t_j$ (i.e., if an eigenvalue begins or ends at a crosssing). If an eigenvalue 
moving in the counterclockwise direction 
arrives at $-1$ at $t = t_j$, then we increment the difference forward, while if 
the eigenvalue arrives at -1 from the clockwise direction we do not (because it
was already in $A_j$ prior to arrival). On
the other hand, suppose an eigenvalue resides at -1 at $t = t_{j-1}$ and moves
in the counterclockwise direction. The eigenvalue remains in $A_j$, and so we do not increment
the difference. However, if the eigenvalue leaves in the clockwise direction 
then we decrement the difference. In summary, the difference increments forward upon arrivals 
in the counterclockwise direction, but not upon arrivals in the clockwise direction,
and it decrements upon departures in the clockwise direction, but not upon 
departures in the counterclockwise direction.      

We are now ready to define the Maslov index.

\begin{definition} \label{dfnDef3.6}  
Let $\mathcal{L} (t) = (\ell_1 (t), \ell_2 (t))$, where $\ell_1, \ell_2:I \to \Lambda (n)$ 
are continuous paths in the Lagrangian--Grassmannian. 
The Maslov index $\mas(\mathcal{L};I)$ is defined by
\begin{equation}
\mas(\mathcal{L};I)=\sum_{j=1}^n(k(t_j,\epsilon_j)-k(t_{j-1},\epsilon_j)).
\end{equation}
\end{definition}

\begin{remark} As discussed in \cite{BF98}, the Maslov index does not depend
on the choices of $\{t_j\}_{j=0}^n$ and $\{\epsilon_j\}_{j=1}^n$, so long as 
they follow the specifications above. 
\end{remark}

One of the most important features of the Maslov index is homotopy invariance, 
for which we need to consider continuously varying families of Lagrangian 
paths. To set some notation, we denote by $\mathcal{P} (I)$ the collection 
of all paths $\mathcal{L} (t) = (\ell_1 (t), \ell_2 (t))$, where 
$\ell_1, \ell_2:I \to \Lambda (n)$ are continuous paths in the 
Lagrangian--Grassmannian. We say that two paths 
$\mathcal{L}, \mathcal{M} \in \mathcal{P} (I)$ are homotopic provided 
there exists a family $\mathcal{H}_s$ so that 
$\mathcal{H}_0 = \mathcal{L}$, $\mathcal{H}_1 = \mathcal{M}$, 
and $\mathcal{H}_s (t)$ is continuous as a map from $(t,s) \in I \times [0,1]$
into $\Lambda (n)$. 
 
The Maslov index has the following properties (see, for example, \cite{HLS} in 
the current setting, or Theorem 3.6 in \cite{F} for a more general result). 

\medskip
\noindent
{\bf (P1)} (Path Additivity) If $a < b < c$ then 
\begin{equation*}
\mas (\mathcal{L};[a, c]) = \mas (\mathcal{L};[a, b]) + \mas (\mathcal{L}; [b, c]).
\end{equation*}

\medskip
\noindent
{\bf (P2)} (Homotopy Invariance) If $\mathcal{L}, \mathcal{M} \in \mathcal{P} (I)$ 
are homotopic, with $\mathcal{L} (a) = \mathcal{M} (a)$ and  
$\mathcal{L} (b) = \mathcal{M} (b)$ (i.e., if $\mathcal{L}, \mathcal{M}$
are homotopic with fixed endpoints) then 
\begin{equation*}
\mas (\mathcal{L};[a, b]) = \mas (\mathcal{M};[a, b]).
\end{equation*}

\section{Application to Schr\"odinger Operators} \label{schrodinger_section}

For $H$ in (\ref{main}), a value $\lambda \in \mathbb{R}$ 
is an eigenvalue (see Definition \ref{spectrum}) if and only if there exist coefficient 
vectors $\alpha (\lambda), \beta (\lambda) \in \mathbb{R}^n$ 
and an eigenfunction $\phi (x; \lambda)$ so that 
$\mathbf{p} = {\phi \choose \phi'}$ satisfies
\begin{equation*}
\mathbf{X}^- (x; \lambda) \alpha (\lambda)
= \mathbf{p} (x; \lambda) = \mathbf{X}^+ (x; \lambda) \beta (\lambda).
\end{equation*}
This clearly holds if and only if the Lagrangian subspaces 
$\ell^- (x; \lambda)$ and $\ell^+ (x; \lambda)$ have 
non-trivial intersection. Moreover, the dimension of intersection 
will correspond with the geometric multiplicity of $\lambda$ 
as an eigenvalue. In this way, we can fix any $x \in \mathbb{R}$ and compute 
the number of negative eigenvalues of $H$, including multiplicities,
by counting the intersections of $\ell^- (x; \lambda)$ and 
$\ell^+ (x; \lambda)$, including multiplicities. Our approach 
will be to choose $x = x_{\infty}$ for a sufficiently large 
value $x_{\infty} > 0$. Our tool for counting the number and 
multiplicity of intersections will be the Maslov index, and 
our two Lagrangian subspaces (in the roles
of $\ell_1$ and $\ell_2$ above) will be $\ell^- (x; \lambda)$
and $\ell^+_{\infty} (\lambda) := \ell^+ (x_{\infty}; \lambda)$. 
We will denote the Lagrangian frame associated
with $\ell^+_{\infty}$ by 
\begin{equation*}
\mathbf{X}_{\infty}^+ (\lambda) 
= {X^+_{\infty} (\lambda) \choose Y^+_{\infty} (\lambda)}.
\end{equation*}

\begin{remark} \label{discontinuous}
We will verify in the appendix that while the limit
\begin{equation*}
\ell^-_{+\infty} (\lambda) := \lim_{x \to +\infty} \ell^- (x; \lambda)
\end{equation*}
is well defined for each $\lambda \le 0$, the resulting
limit is not necessarily continuous as a function of $\lambda$. This 
is our primary motivation for working with $x_{\infty}$ rather than 
with the asymptotic limit.
\end{remark}

Our analysis will be based on computing the Maslov index along a 
closed path in the $x$-$\lambda$ plane, determined by sufficiently large 
values $x_{\infty}, \lambda_{\infty} > 0$. First, if we fix $\lambda = 0$ and 
let $x$ run from $-\infty$ to $x_{\infty}$, we denote the 
resulting path $\Gamma_0$ (the {\it right shelf}). Next, 
we fix $x = x_{\infty}$ and let $\Gamma_+$ denote a path in 
which $\lambda$ decreases from $0$ to $-\lambda_{\infty}$.  
Continuing counterclockwise along our path, we denote by 
$\Gamma_{\infty}$ the path obtained by fixing 
$\lambda = -\lambda_{\infty}$ and letting $x$ run from $x_{\infty}$
to $-\infty$ (the {\it left shelf}). Finally, we close the path 
in an asymptotic sense by taking a final path, $\Gamma_-$, 
with $\lambda$ running from $-\lambda_{\infty}$ to $0$ 
(viewed as the asymptotic limit as $x \to - \infty$; we refer 
to this as the {\it bottom shelf}). See Figure \ref{box_figure}.

\begin{figure}[ht] 
\begin{center}\includegraphics[%
  width=12cm,
  height=8cm]{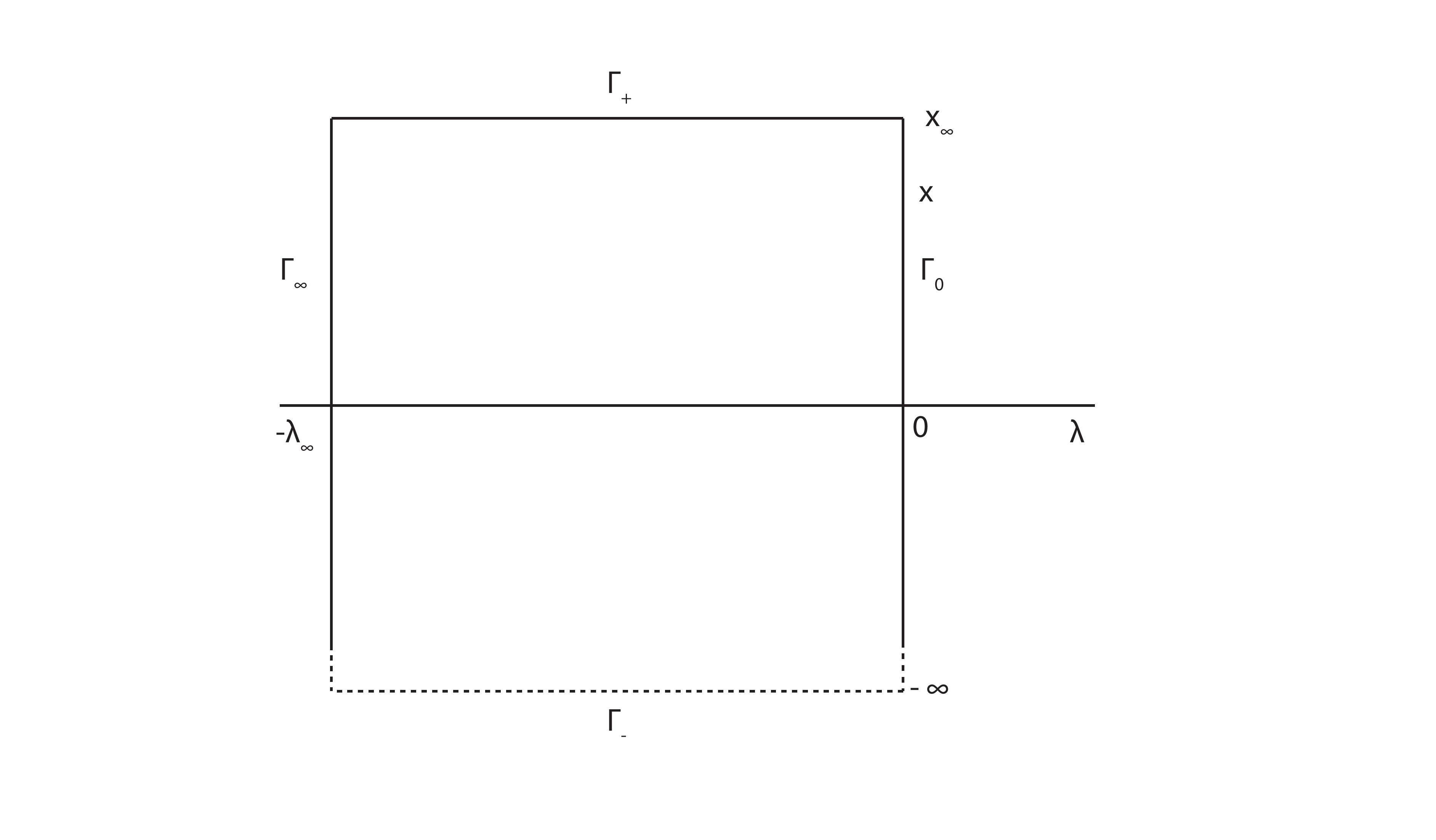}\end{center}
\caption{Maslov Box. \label{box_figure}}
\end{figure}

We recall that we can take the vectors in our frame $\mathbf{X}^- (x; \lambda)$
to be 
\begin{equation*}
\scripty{r}_{\,n+j}^{\,-} + \mathbf{E}_{n+j}^-,
\end{equation*}
from which we see that $\ell^- (x;\lambda)$ approaches the asymptotic frame
$\mathbf{R}^- (\lambda)$ as $x \to -\infty$. Introducing the change of 
variables 
\begin{equation*}
x = \ln (\frac{1+\tau}{1-\tau}) \iff \tau = \frac{e^x - 1}{e^x +1},
\end{equation*} 
we see that $\ell^-$ can be viewed as a continuous map on the compact 
domain 
\begin{equation*}
[-1, \frac{e^{x_\infty} - 1}{e^{x_\infty}+1}] \times [-\lambda_{\infty}, 0].
\end{equation*}

In the setting of (\ref{main}), our evolving Lagrangian subspaces have
frames $\mathbf{X}^- (x; \lambda)$ and $\mathbf{X}^+_{\infty} (\lambda)$,
so that $\tilde{W}$ from (\ref{tildeW}) becomes 
\begin{equation} \label{tildeWmain}
\begin{aligned}
\tilde{W} (x; \lambda) 
& = - (X^- (x; \lambda) + i Y^- (x; \lambda)) (X^- (x; \lambda) - i Y^- (x; \lambda))^{-1} \\
&\quad \times (X^+_{\infty}(\lambda) - i Y^+_{\infty} (\lambda)) (X^+_{\infty} (\lambda) + i Y^+_{\infty} (\lambda))^{-1}. 
\end{aligned}
\end{equation} 
Since $\tilde{W} (x; \lambda)$ is unitary, its eigenvalues are confined to 
the until circle in $\mathbb{C}$, $S^1$. In the limit as $x \to -\infty$ we 
obtain 
\begin{equation} \label{tildeWmainInf}
\begin{aligned}
\tilde{W}^- (\lambda) := \lim_{x \to -\infty} \tilde{W} (x; \lambda)
& = - (R^- + i S^- (\lambda)) (R^- - i S^- (\lambda))^{-1} \\
& \quad \times (X^+_{\infty}(\lambda) - i Y^+_{\infty} (\lambda)) (X^+_{\infty} (\lambda) + i Y^+_{\infty} (\lambda))^{-1}. 
\end{aligned}
\end{equation}

\subsection{Monotonicity} \label{monotonicity_section}

Our first result for this section 
asserts that the eigenvalues of $\tilde{W} (x;\lambda)$ and $\tilde{W}^- (\lambda)$
rotate monotonically as $\lambda$ varies 
along $\mathbb{R}$. In order to prove this, we will use a lemma from 
\cite{HS}, which we state as follows (see also Theorem V.6.1 in \cite{At}).

\begin{lemma} [\cite{HS}, Lemma 3.11.] 
Let $\tilde{W} (\tau)$ be a $C^1$ family of unitary $n \times n$ matrices on 
some interval $I$, satisfying a differential equation 
$\frac{d}{d\tau} \tilde{W} (\tau) = i \tilde{W} (\tau) \tilde{\Omega} (\tau)$, 
where $\tilde{\Omega} (\tau)$ is a continuous, self-adjoint and negative-definite 
$n \times n$ matrix. 
Then the eigenvalues of $\tilde{W} (\tau)$ move (strictly) monotonically clockwise on the 
unit circle as $\tau$ increases.  
\label{HS_monotonicity} 
\end{lemma} 

We are now prepared to state and prove our monotonicity lemma.

\begin{lemma} \label{monotonicity_lemma}
Let $V \in C(\mathbb{R})$ be a real-valued symmetric matrix, and suppose (A1)
and (A2) hold. Then for each fixed $x \in \mathbb{R}$
the eigenvalues of $\tilde{W} (x; \lambda)$ rotate monotonically clockwise
as $\lambda \in (-\infty, \nu_{\min})$ increases. Moreover, the eigenvalues 
of $\tilde{W}^{-} (\lambda)$ rotate (strictly) monotonically clockwise as 
$\lambda \in (-\infty, \nu_{\min})$ increases. 
\end{lemma}

\begin{remark} \label{monotonicity_remark}
The monotoncity described in Lemma \ref{monotonicity_lemma} seems to be 
generic for self-adjoint operators in a broad range of settings (see, for 
example, \cite{HS}); monotonicity in $x$ is not generic.    
\end{remark}

\begin{proof} Following \cite{HS}, we begin by computing 
$\frac{\partial \tilde{W}}{\partial \lambda}$, and for this 
calculation it's convenient to write 
$\tilde{W} (x; \lambda) = - \tilde{W}_1 (x; \lambda) \tilde{W}_2 (\lambda)$,
where 
\begin{equation*}
\begin{aligned}
\tilde{W}_1 (x; \lambda) &= 
(X^- (x;\lambda) + iY^- (x;\lambda)) (X^- (x;\lambda) - iY^- (x;\lambda))^{-1} \\
\tilde{W}_2 (\lambda) &=
(X^+_{\infty} (\lambda) - i Y^+_{\infty} (\lambda)) 
(X^+_{\infty} (\lambda) + i Y^+_{\infty} (\lambda))^{-1}.
\end{aligned}
\end{equation*}
For $\tilde{W}_1$, we have (suppressing independent variables for notational brevity) 
\begin{equation*}
\begin{aligned}
\frac{\partial \tilde{W}_1}{\partial \lambda} &=
(X^-_{\lambda} + iY^-_{\lambda}) (X^- - iY^-)^{-1} 
- (X^- + iY^-) (X^- - iY^-)^{-1} (X^-_{\lambda} - iY^-_{\lambda}) (X^- - iY^-)^{-1} \\
&= (X^-_{\lambda} + iY^-_{\lambda}) (X^- - iY^-)^{-1} 
- \tilde{W}_1 (X^-_{\lambda} - iY^-_{\lambda}) (X^- - iY^-)^{-1}.  
\end{aligned}
\end{equation*}
If we multiply by $\tilde{W}_1^*$ we find 
\begin{equation*}
\begin{aligned}
\tilde{W}_1^* \frac{\partial \tilde{W}_1}{\partial \lambda}
&= ({X^-}^t + i {Y^-}^t)^{-1} ({X^-}^t - i{Y^-}^t)  (X^-_{\lambda} + iY^-_{\lambda}) (X^- - iY^-)^{-1} \\
&- (X^-_{\lambda} - iY^-_{\lambda}) (X^- - iY^-)^{-1} \\
&= ({X^-}^t + i {Y^-}^t)^{-1} \Big{\{} ({X^-}^t - i{Y^-}^t)  (X^-_{\lambda} + iY^-_{\lambda}) \\
& \quad \quad - ({X^-}^t + i {Y^-}^t) (X^-_{\lambda} - iY^-_{\lambda})  \Big{\}} (X^- - iY^-)^{-1} \\
&= \Big( ({X^-} - i {Y^-})^{-1} \Big)^* 
\Big{\{} 2i {X^-}^t Y^-_{\lambda} - 2i {Y^-}^t X^-_{\lambda} \Big{\}} (X^- - iY^-)^{-1}.
\end{aligned}
\end{equation*}
Multiplying back through by $\tilde{W}_1$, we conclude 
\begin{equation*}
\frac{\partial \tilde{W}_1}{\partial \lambda} = i \tilde{W}_1 \tilde{\Omega}_1,
\end{equation*}
where 
\begin{equation*}
\tilde{\Omega}_1 = \Big( ({X^-} - i {Y^-})^{-1} \Big)^* 
\Big{\{} 2 {X^-}^t Y^-_{\lambda} - 2 {Y^-}^t X^-_{\lambda} \Big{\}} \Big( (X^- - iY^-)^{-1} \Big).
\end{equation*}

Likewise, we find that 
\begin{equation*}
\frac{\partial \tilde{W}_2}{\partial \lambda} = i \tilde{W}_2 \tilde{\Omega}_2,
\end{equation*}
where 
\begin{equation} \label{tildeomega}
\tilde{\Omega}_2 = \Big( (X^+_{\infty} + i Y^+_{\infty})^{-1} \Big)^* 
\Big{\{} 2 {Y^+_{\infty}}^t \partial_{\lambda} X^+_{\infty} 
- 2 {X^+_{\infty}}^t \partial_{\lambda} Y^+_{\infty} \Big{\}} \Big( (X^+_{\infty} + iY^+_{\infty})^{-1} \Big).
\end{equation}

Combining these observations, we find 
\begin{equation} \label{combined1}
\begin{aligned}
\frac{\partial \tilde{W}}{\partial \lambda} &= - \frac{\partial \tilde{W}_1}{\partial \lambda} \tilde{W}_2
- \tilde{W}_1 \frac{\partial \tilde{W}_2}{\partial \lambda} 
=  - i \tilde{W}_1 \tilde{\Omega}_1 \tilde{W}_2 - i \tilde{W}_1 \tilde{W}_2 \tilde{\Omega}_2 \\
&= - i \tilde{W}_1 \tilde{W}_2 (\tilde{W}_2^* \tilde{\Omega}_1 \tilde{W}_2) - i \tilde{W}_1 \tilde{W}_2 \tilde{\Omega}_2 
= i \tilde{W} \tilde{\Omega},
\end{aligned}
\end{equation}
where (recalling that $\tilde{W} = - \tilde{W}_1 \tilde{W}_2$)
\begin{equation*}
\tilde{\Omega} = \tilde{W}_2^* \tilde{\Omega}_1 \tilde{W}_2 + \tilde{\Omega}_2.
\end{equation*}

We see that the behavior of $\frac{\partial \tilde{W}}{\partial \lambda}$ will be determined by 
the quantities ${X^-}^t Y^-_{\lambda} - {Y^-}^t X^-_{\lambda}$ and 
${Y^+_{\infty}}^t \partial_{\lambda} X^+_{\infty} 
- {X^+_{\infty}}^t \partial_{\lambda} Y^+_{\infty}$. 
For the former, we differentiate with respect to $x$ to find 
\begin{equation*}
\begin{aligned}
\frac{\partial}{\partial x} 
\Big{\{} {X^-}^t Y^-_{\lambda} - {Y^-}^t X^-_{\lambda} \Big{\}}
&= 
{X^-}^t_x Y^-_{\lambda} + {X^-}^t Y^-_{\lambda x} 
- {Y^-}^t_x X^-_{\lambda} - {Y^-}^t X^-_{\lambda x} \\
&=
{Y^-}^t Y^-_{\lambda} + {X^-}^t (V X^- - \lambda X^-)_{\lambda}
- (VX^- - \lambda X^-)^t X^-_{\lambda} - {Y^-}^t Y^-_{\lambda} \\
&= - {X^-}^t X^-,
\end{aligned}
\end{equation*}
where we've used $X^-_x = Y^-$ and $Y^-_x  = V(x) X^- - \lambda X^-$.
Integrating from $- \infty$ to $x$, we find 
\begin{equation*}
{X^-}^t Y^-_{\lambda} - {Y^-}^t X^-_{\lambda} =
- \int_{-\infty}^x {X^-}^t (y;\lambda) X^- (y;\lambda) dy,
\end{equation*}
from which it is clear that ${X^-}^t Y^-_{\lambda} - {Y^-}^t X^-_{\lambda}$ is
negative definite, which implies that $\tilde{\Omega}_1$ is negative 
definite. 

Likewise, even though $x_{\infty}$ is fixed, we can differentiate 
\begin{equation*}
{Y^+ (x; \lambda)}^t X^+_{\lambda} (x; \lambda) 
- {X^+ (x; \lambda)}^t Y^+_{\lambda} (x; \lambda)
\end{equation*}
with respect to 
$x$ and evaluate at $x = x_{\infty}$ to find 
\begin{equation*}
{Y^+_{\infty}}^t \partial_{\lambda} X^+_{\infty} 
- {X^+_{\infty}}^t \partial_{\lambda} Y^+_{\infty} 
= - \int_{x_{\infty}}^{+\infty} {X^-}^t (y;\lambda) X^- (y;\lambda) dy,
\end{equation*}
from which it is clear that ${Y^+_{\infty}}^t \partial_{\lambda} X^+_{\infty} 
- {X^+_{\infty}}^t \partial_{\lambda} Y^+_{\infty}$ is
negative definite, which implies that $\tilde{\Omega}_2$ is negative 
definite. 

We conclude that $\tilde{\Omega}$ is negative definite, at which point we can 
employ Lemma 3.11 from \cite{HS} to obtain the claim. 

For the case of $\tilde{W}^- (\lambda)$, we have 
$\tilde{W}^- (\lambda) = - \tilde{W}_1 (\lambda) \tilde{W}_2 (\lambda)$,
where 
\begin{equation*}
\tilde{W}_1 (\lambda) = 
({R^-} + i{S^-}) (R^- - iS^-)^{-1},
\end{equation*}
and $\tilde{W}_2 (\lambda)$ is as above. Computing as before, we find 
\begin{equation*}
\frac{\partial \tilde{W}_1}{\partial \lambda} = i \tilde{W}_1 \tilde{\Omega}_1,
\end{equation*}
where in this case
\begin{equation*}
\tilde{\Omega}_1 = \Big( ({R^-} - i {S^-})^{-1} \Big)^* 
\Big{\{} 2 {R^-}^t S^-_{\lambda} - 2 {S^-}^t R^-_{\lambda} \Big{\}} \Big( (R^- - iS^-)^{-1} \Big).
\end{equation*} 
Recalling that $R^-_{\lambda} = 0$, we see that the nature of $\tilde{\Omega}_1$ is 
determined by ${R^-}^t S^-_{\lambda}$. Recalling that 
\begin{equation*}
S^- (\lambda) = 
\begin{pmatrix}
\mu_{n+1}^- (\lambda) r_1^- & \mu_{n+2}^- (\lambda) r_{2}^- & \dots & \mu_{2n}^- (\lambda) r_n^-
\end{pmatrix},
\end{equation*}
we have (recalling $\mu_{n+j}^+ (\lambda) = \sqrt{\nu_{j}^- - \lambda}\,$)
\begin{equation*}
S^-_{\lambda} (\lambda) = 
- \frac{1}{2} \begin{pmatrix}
\frac{1}{\mu_{n+1}^- (\lambda)} r_1^- & \frac{1}{\mu_{n+2}^- (\lambda)} r_{2}^+ & \dots & \frac{1}{\mu_{2n}^- (\lambda)} r_n^-
\end{pmatrix}.
\end{equation*}
In this way, orthogonality of the $\{r_j^-\}_{j=1}^n$ leads to the relation 
\begin{equation} \label{RSlambda2}
{R^-}^t S^-_{\lambda} = -\frac{1}{2} 
\begin{pmatrix}
\frac{1}{\mu_{n+1}^- (\lambda)} & 0 & \dots & 0 \\
0& \frac{1}{\mu_{n+2}^- (\lambda)} & \dots & 0 \\
\vdots & \vdots & \vdots & \vdots \\
0 & 0 & \dots & \frac{1}{\mu_{2n}^- (\lambda)}
\end{pmatrix}.
\end{equation}
Since the $\{\mu_{n+j}^-\}_{j=1}^n$ are all positive (for $\lambda < \nu_{\min}$), 
we see that $\tilde{\Omega}_1$ is self-adjoint and negative definite. 

The matrix $\tilde{W}_2$ is unchanged, so we can draw the same 
conclusion about monotonicity. 
\end{proof}

\subsection{Lower Bound on the Spectrum of $H$} \label{bound_section}

We have already seen that if the eigenvalues of $V_{\pm}$ are all non-negative
then the essential spectrum of $H$ is bounded below by 0. In fact, it's 
bounded below by the smallest eigenvalue of the two matrices $V_{\pm}$. For 
the point spectrum, if $\lambda$ is an eigenvalue of $H$ then there 
exists a corresponding eigenfunction $\phi (\cdot; \lambda) \in H^1 (\mathbb{R})$.
If we take an $L^2 (\mathbb{R})$ inner product of (\ref{main}) with 
$\phi$ we find 
\begin{equation*}
\lambda \|\phi\|_2^2 = \|\phi'\|_2^2 + \langle V\phi, \phi\rangle
\ge - C \|\phi\|_2^2,
\end{equation*}
for some contant $C > 0$ taken so that 
$|\langle V\phi,\phi\rangle| \le C \|\phi\|_2^2$ for all 
$\phi \in H^1 (\mathbb{R})$. 
We conclude that $\sigma_{pt} (H) \subset [-C, \infty)$. For example, 
$C = \|V\|_{\infty}$ clearly works. In what follows, 
we will take a value $\lambda_{\infty}$ sufficiently large, and in 
particular we will take $\lambda_{\infty} > C$ (additional requirements 
will be added as well, but they can all be accommodated by taking 
$\lambda_{\infty}$ larger, so that this initial restriction continues
to hold).

\subsection{The Top Shelf} \label{top_section}

Along the top shelf $\Gamma_+$, the Maslov index counts intersections of the Lagrangian
subspaces $\ell^- (x_{\infty}; \lambda)$ and 
$\ell^+_{\infty} (\lambda) = \ell^+ (x_{\infty}; \lambda)$.
Such intersections will correspond with solutions of (\ref{main}) that 
decay at both $\pm \infty$, and hence will correspond with eigenvalues. 
Moreover, the dimension of these intersections will correspond with 
the dimension of the space of solutions that decay at both $\pm \infty$,
and so will correspond with the geometric multiplicity of the eigenvalues. 
Finally, we have seen that the eigenvalues of $\tilde{W} (x; \lambda)$
rotate monotonically counterclockwise as $\lambda$ decreases from $0$ 
to $-\lambda_{\infty}$ (i.e., as $\Gamma_+$ is traversed), and so the 
Maslov index on $\Gamma_+$ is a direct count of the crossings, including
multiplicity (with no cancellations arising from crossings in opposite
directions). We conclude that the Maslov index associated with this path 
will be a count, including multiplicity, of the negative eigenvalues of 
$H$; i.e., of the Morse index. We can express 
these considerations as  
\begin{equation*}
\Mor(H) = \Mas(\ell^-, \ell^+_{\infty}; \Gamma_+).
\end{equation*}

\subsection{The Bottom Shelf} \label{bottom_section}

For the bottom shelf, we have 
\begin{equation} \label{tildeWbottom1}
\begin{aligned}
\tilde{W}^- (\lambda) &= - (R^- + i S^- (\lambda)) (R^- - i S^- (\lambda))^{-1} \\
&\times (X^+_{\infty} (\lambda) - i Y^+_{\infty} (\lambda)) 
(X^+_{\infty} (\lambda) + i Y^+_{\infty} (\lambda))^{-1}.
\end{aligned}
\end{equation}  
By choosing $x_{\infty}$ suitably large, we can ensure that the frame $\mathbf{X}^+_{\infty} (\lambda)$
is as close as we like to the frame $\mathbf{R}^+ (\lambda)$, where we recall
$\mathbf{R}^- = {R^- \choose S^-}$ and $\mathbf{R}^+ = {R^+ \choose S^+}$. (As noted 
in Remark \ref{discontinuous} $\ell^-_{+\infty} (\lambda)$ is not necessarily 
continuous in $\lambda$, but $\ell^+_{\mathbf{R}} (\lambda)$ certainly is 
continuous in $\lambda$.) We will 
proceed by analyzing the matrix 
\begin{equation} \label{tildeWbottom}
\tilde{\mathcal{W}}^- (\lambda) := - (R^- + i S^- (\lambda)) (R^- - i S^- (\lambda))^{-1} (R^+ - i S^+ (\lambda)) (R^+ + i S^+ (\lambda))^{-1},
\end{equation}  
for which we will be able to conclude that for $\lambda < \nu_{\min}$, 
$-1$ is never an eigenvalue. By continuity, we will be able to draw conclusions
about $\tilde{W}^- (\lambda)$ as well.

\begin{lemma} \label{bottom_lemma}
For any $\lambda < \nu_{\min}$ the spectrum of  
$\tilde{\mathcal{W}}^- (\lambda)$ does not include $-1$.
\end{lemma}

\begin{proof} 
We need only show that for any $\lambda < \nu_{\min}$
the $2n$ vectors comprising the columns of $\mathbf{R}^-$ and $\mathbf{R}^+$
are linearly independent. We proceed by induction, first establishing that
any single column of $\mathbf{R}^-$ is linearly independent of the columns
of $\mathbf{R}^+$. Suppose not. Then there is some $j \in \{1, 2, \dots, n\}$,
along with some collection of constants $\{c_k\}_{k=1}^n$ so that 
\begin{equation} \label{setup1}
\scripty{r}_{\,n+j}^{\,-} = \sum_{k=1}^n c_k \scripty{r}_{\,k}^{\,+}.
\end{equation} 
Recalling the definitions of $\scripty{r}_{\,n+j}^{\,-}$ and $\scripty{r}_{\,k}^{\,+}$,
we have the two equations 
\begin{equation*}
\begin{aligned}
r_j^- &= \sum_{k=1}^n c_k r^+_{n+1-k} \\
\mu_{n+j}^- r_j^- &= \sum_{k=1}^n c_k \mu_k^+ r^+_{n+1-k}.
\end{aligned}
\end{equation*}
Multiplying the first of these equations by $\mu_{n+j}^-$, 
and subtracting the second equation from the result, we find
\begin{equation*}
0 = \sum_{k=1}^n (\mu_{n+j}^- - \mu_k^+) c_k r^+_{n+1-k}.
\end{equation*}
Since the collection $\{r^+_{n+1-k}\}_{k=1}^n$ is linearly 
independent, and since $\mu_{n+j}^- - \mu_k^+ > 0$ for all 
$k \in \{1, 2, \dots, n\}$ (for $\lambda < \nu_{\min}$), 
we conclude that the constants 
$\{c_k\}_{k=1}^n$ must all be zero, but this contradicts 
(\ref{setup1}). 

For the induction step, suppose that for some $1 \le m < n$, 
any $m$ elements of the collection $\{\scripty{r}_{\,n+j}^{\,-}\}_{j=1}^n$
are linearly independent of the set $\{ \scripty{r}_{\,k}^{\,+} \}_{k=1}^n$.
We want to show that any $m+1$ elements of the collection 
$\{\scripty{r}_{\,n+j}^{\,-}\}_{j=1}^n$
are linearly independent of the set $\{ \scripty{r}_{\,k}^{\,+} \}_{k=1}^n$.
If not, then by a change of labeling if necessary there exist constants
$\{c_l^-\}_{l=2}^{m+1}$ and $\{c_k^+\}_{k=1}^n$ so that  
\begin{equation} \label{setup2}
\scripty{r}_{\,n+1}^{\,-} = \sum_{l=2}^{m+1} c_l^- \scripty{r}_{\,n+l}^{\,-}
+ \sum_{k=1}^n c_k^+ \scripty{r}_{\,k}^{\,+}.
\end{equation} 
Again, we have two equations 
\begin{equation*}
\begin{aligned}
r_1^- &= \sum_{l=2}^{m+1} c_l^- r_l^- + \sum_{k=1}^n c_k^+ r_{n+1-k}^+ \\
\mu^-_{n+1} r_1^- &= \sum_{l=2}^{m+1} c_l^- \mu^-_{n+l} r_l^- 
+ \sum_{k=1}^n c_k^+ \mu_k^+ r_{n+1-k}^+.
\end{aligned}
\end{equation*}
Multiplying the first of these equations by $\mu^-_{n+1}$, and 
subtracting the second equation from the result, we obtain the 
relation 
\begin{equation*}
0 = \sum_{l=2}^{m+1} c_l^- (\mu^-_{n+1} - \mu^-_{n+l}) r_l^- 
+ \sum_{k=1}^n c_k^+ (\mu^-_{n+1} - \mu_k^+) r_{n+1-k}^+.
\end{equation*}
By our induction hypothesis, the vectors on the right-hand side
are all linearly independent, and since $\mu^-_{n+1} - \mu_k^+ > 0$
for all $k \in {1, 2, \dots, n}$, we can conclude that 
$c_k^+ = 0$ for all $k \in {1, 2, \dots, n}$. (Notice that we
make no claim about the $c_l^-$.) Returning to (\ref{setup2}),
we obtain a contradiction to the linear independence of 
the collection $\{\scripty{r}_{\,n+j}^{\,-}\}_{j=1}^n$.

Continuing the induction up to $m = n-1$ gives the claim for 
$\lambda < \nu_{\min}$. 
\end{proof}

\begin{remark} It is important to note that we do not include the 
case $\lambda = \nu_{\min}$ in our lemma, and indeed the lemma 
does not generally hold in this case. For example, consider the 
case in which $V(x)$ vanishes identically at both $\pm \infty$
(i.e., $V_- = V_+ = 0$). In this case, we can take $R^- = I$, 
$S^- = \sqrt{-\lambda} I$, $R^+ = \check{I}$, and 
$S^+ = -\sqrt{-\lambda} \check{I}$, where 
\begin{equation*}
\check{I} = 
\begin{pmatrix}
0 & 0 & \dots & 1 \\
0 & 0 & 1 & 0 \\
\vdots & \vdots & \vdots &\vdots \\
1 & 0 & \dots & 0  
\end{pmatrix}.
\end{equation*} 
We easily find 
\begin{equation*}
\tilde{\mathcal{W}}^- (\lambda) = - \frac{(1+i\sqrt{-\lambda})^2}{(1-i\sqrt{-\lambda})^2} I,
\end{equation*}
and we see explicitly that $\tilde{W}^- (0) = - I$, so that all $n$ eigenvalues
reside at $-1$. Moreover, as $\lambda$ proceeds from 0 toward $- \infty$ the 
eigenvalues of $\tilde{W}^- (\lambda)$ remain coalesced, and move monotonically 
counterclockwise around $S^1$, returning to $-1$ in the limit as $\lambda \to - \infty$.
In this case, we can conclude that for the path from $0$ to $-\lambda_{\infty}$,
the Maslov index does not increment. 
\end{remark}

We immediately obtain the following lemma. 

\begin{lemma} \label{bottom_shelf_case1}   
Let $V \in C(\mathbb{R})$ be a real-valued symmetric matrix, and suppose (A1)
and (A2) hold. If 
\begin{equation*}
\dim (\ell^-_{\mathbf{R}} (0) \cap \ell^+_{\mathbf{R}} (0)) = 0
\end{equation*}
then we can choose $x_{\infty}$ sufficiently large so that we will have 
\begin{equation*}
\dim (\ell^-_{\mathbf{R}} (\lambda) \cap \ell^+_{\infty} (\lambda)) = 0
\end{equation*}
for all $\lambda \in [-\lambda_{\infty}, 0]$. It follows that in 
this case
\begin{equation*}
\Mas (\ell^-, \ell^+_{\infty}; \Gamma_-) = 0.
\end{equation*}
Moreover, if 
\begin{equation*}
\dim (\ell^-_{\mathbf{R}} (0) \cap \ell^+_{\mathbf{R}} (0)) \ne 0
\end{equation*}
then given any $\lambda_0$ with $0 < \lambda_0 < \lambda_{\infty}$
we can take $x_{\infty}$ sufficiently large so that  
\begin{equation*}
\dim (\ell^-_{\mathbf{R}} (\lambda) \cap \ell^+_{\infty} (\lambda)) = 0
\end{equation*}
for all $\lambda \in [-\lambda_{\infty}, -\lambda_0]$.
\end{lemma}

\begin{proof} First, if 
\begin{equation*}
\dim (\ell^-_{\mathbf{R}} (0) \cap \ell^+_{\mathbf{R}} (0)) = 0
\end{equation*}
then none of the eigenvalues of $\tilde{\mathcal{W}}^- (0)$ is $-1$,
and so according to Lemma \ref{bottom_lemma}, none of the 
eigenvalues of $\tilde{\mathcal{W}}^- (\lambda)$ is $-1$ for 
any $\lambda \in [-\lambda_{\infty}, 0]$. In particular, 
since the interval $[-\lambda_{\infty}, 0]$ is compact
there exists some $\epsilon > 0$ so that each eigenvalue 
$\tilde{\omega} (\lambda)$ of $\tilde{\mathcal{W}}^- (\lambda)$ 
satisfies 
\begin{equation*}
|\tilde{\omega} (\lambda) + 1| > \epsilon
\end{equation*}
for all $\lambda \in [-\lambda_{\infty}, 0]$. 

Similarly as above, we can make the change of variables 
\begin{equation*}
x_{\infty} = \ln (\frac{1+\tau_{\infty}}{1-\tau_{\infty}}),
\iff
\tau_{\infty} = \frac{e^{x_{\infty}}-1}{e^{x_{\infty}} + 1}.
\end{equation*}
This allows us to view $\tilde{W}^-$ as a continuous function on 
the compact domain $(x_{\infty},\lambda) \in [1-\delta, 1] \times [-\lambda_{\infty}, 0]$,
where $\delta > 0$ is small, indicating that $x_{\infty}$ is taken 
to be large.
We see that $\tilde{W}^-$ is uniformly continuous and so by choosing 
$\tau_{\infty}$ sufficiently close to 1, we can force the eigenvalues
of $\tilde{W}^-$ to be as close to the eigenvalues of 
$\tilde{\mathcal{W}}^- (\lambda)$ as we like. We take $\tau_{\infty}$
sufficiently close to 1 so that for each $\lambda \in [-\lambda_{\infty}, 0]$
and each eigenvalue $\tilde{\omega}$ of $\tilde{\mathcal{W}}^- (\lambda)$ 
there is a corresponding eigenvalue of $\tilde{W}^-$, which we denote $\omega (\lambda)$ 
so that $|\tilde{\omega} (\lambda) - \omega (\lambda)| < \epsilon/2$. But
then 
\begin{equation*}
\begin{aligned}
\epsilon &< |\tilde{\omega} (\lambda)+1| = |\tilde{\omega} (\lambda) - \omega (\lambda)
+ \omega (\lambda) + 1| \\
&\le 
|\tilde{\omega} (\lambda) - \omega (\lambda)| + |\omega (\lambda) + 1|
< \frac{\epsilon}{2} + |\omega (\lambda) + 1|, 
\end{aligned}
\end{equation*}
from which we conclude that 
\begin{equation*}
|\omega (\lambda) + 1| \ge \frac{\epsilon}{2},
\end{equation*}   
for all $\lambda \in [-\lambda_{\infty}, 0]$.

For the {\it Moreover} claim, we simply replace $[-\lambda_{\infty}, 0]$
with $[-\lambda_{\infty}, -\lambda_0]$ in the above argument.
\end{proof}

\subsection{The Left Shelf} \label{left_section}

For the left shelf $\Gamma_{\infty}$, we need to understand the Maslov index associated with 
$\tilde{W} (x; - \lambda_{\infty})$ (with $\lambda_{\infty}$ sufficiently large) 
as $x$ goes from $-\infty$ to $x_{\infty}$ (keeping in mind that the path $\Gamma_{\infty}$ 
reverses this flow). In order to accomplish this, we follow the approach of 
\cite{GZ, ZH} in developing large-$|\lambda|$ estimates on solutions 
of (\ref{main}), uniformly in $x$. For $\lambda < 0$, we set 
\begin{equation*}
\xi = \sqrt{-\lambda} x; \quad \phi (\xi) = y(x), 
\end{equation*} 
so that (\ref{main}) becomes 
\begin{equation*}
\phi'' (\xi) + \frac{1}{\lambda} V (\frac{\xi}{\sqrt{-\lambda}}) \phi = \phi.
\end{equation*}
Setting $\Phi_1 = \phi$, $\Phi_2 = \phi'$, 
and $\Phi = {\Phi_1 \choose \Phi_2} \in \mathbb{R}^{2n}$, we can 
express this equation as 
\begin{equation*}
\Phi' = \mathbb{A} (\xi; \lambda) \Phi; \quad 
\mathbb{A} (\xi; \lambda) =
\begin{pmatrix}
0 & I \\
I - \frac{1}{\lambda} V (\frac{\xi}{\sqrt{-\lambda}}) & 0
\end{pmatrix}.
\end{equation*}

We begin by looking for solutions that decay as $x \to - \infty$ (and so 
as $\xi \to - \infty$); i.e., we begin by constructing the frame 
$\mathbf{X}^- (x; - \lambda_{\infty})$. It's convenient to write 
\begin{equation*}
\mathbb{A} (\xi; \lambda) = \mathbb{A}_- (\lambda) + \mathbb{E}_- (\xi; \lambda),
\end{equation*} 
where 
\begin{equation*}
\mathbb{A}_- (\lambda) = 
\begin{pmatrix}
0 & I \\
I - \frac{1}{\lambda} V_- & 0
\end{pmatrix};
\quad 
\mathbb{E}_- (\xi; \lambda) = 
\begin{pmatrix}
0 & 0 \\
\frac{1}{\lambda} (V_- - V(\frac{\xi}{\sqrt{-\lambda}})) & 0
\end{pmatrix}.
\end{equation*}

Fix any $M \gg 0$ and note that according to (A1), we have 
\begin{equation*}
\begin{aligned}
\int_{-\infty}^M |\mathbb{E}_- (\xi; \lambda)| d\xi 
& \le \frac{1}{|\lambda|} \int_{-\infty}^M |V(\frac{\xi}{\sqrt{-\lambda}}) - V_-| d\xi \\
& = \frac{1}{|\lambda|} \int_{-\infty}^{\frac{M}{\sqrt{-\lambda}}} 
|V(x) - V_-| \sqrt{-\lambda} dx \le \frac{K}{\sqrt{-\lambda}}, 
\end{aligned}
\end{equation*}
for some constant $K = K(M)$. Recalling that we are denoting
the eigenvalues of $V_-$ by $\{\nu_j^-\}_{j=1}^n$, we readily check
that the eigenvalues of $\mathbb{A}_- (\lambda)$ can be 
expressed as 
\begin{equation*}
\begin{aligned}
\hat{\mu}_j^- (\lambda) &= - \sqrt{1 - \frac{\nu^-_{m+1-j}}{\lambda}} = \frac{1}{\sqrt{-\lambda}} \mu_j^- \\
\hat{\mu}_{n+j}^- (\lambda) &= \sqrt{1 - \frac{\nu_j^-}{\lambda}} = \frac{1}{\sqrt{-\lambda}} \mu_{n+j}^-,
\end{aligned}
\end{equation*} 
for $j = 1,2, \dots, n$ (ordered, as usual, so that $j < k$ implies 
$\hat{\mu}_j^- \le \hat{\mu}_k^-$). In order to select a solution decaying with 
rate $\hat{\mu}_{n+j}^-$ (as $\xi \to -\infty$), we look for solutions 
of the form 
$\Phi (\xi; \lambda) = e^{\hat{\mu}_{n+j}^- (\lambda) \xi} Z (\xi; \lambda)$,
for which $Z$ satisfies 
\begin{equation*}
Z' = (\mathbb{A}_- (\lambda) - \hat{\mu}_{m+j} (\lambda) I) Z 
+ \mathbb{E}_- (\xi; \lambda) Z.
\end{equation*} 

Proceeding similarly as in the proof of Lemma \ref{ODElemma}, we obtain 
a collection of solutions
\begin{equation*}
Z_{n+j}^- (\xi; \lambda) =  
\hat{\scripty{r}}_{\,n+j}^{\,-} 
+ \mathbf{O} (|\lambda|^{-1/2}), 
\end{equation*}
which lead to 
\begin{equation*}
\Phi_{n+j}^- (\xi; \lambda) =  e^{\hat{\mu}_{n+j}^- (\lambda) \xi}
(\hat{\scripty{r}}_{\,n+j}^{\,-} 
+ \mathbf{O} (|\lambda|^{-1/2})), 
\end{equation*}
where $\hat{\scripty{r}}$ corresponds with $\scripty{r}$, with 
$\mu$ is replaced by $\hat{\mu}$.
Returning to original coordinates, we construct the frame $\mathbf{X}^- (x;\lambda)$
out of basis elements
\begin{equation*}
\begin{pmatrix}
y (x) \\ y' (x)
\end{pmatrix}
= e^{\sqrt{-\lambda} \hat{\mu}_{n+j}^- (\lambda) x}
\Big(
\begin{pmatrix}
r_j^- \\ \sqrt{-\lambda} \hat{\mu}_{n+j}^- r_j^- 
\end{pmatrix}
+
\begin{pmatrix}
\mathbf{O} (|\lambda|^{-1/2}) \\ \mathbf{O} (1)
\end{pmatrix}
\Big).
\end{equation*}

Recalling that when specifying a frame for $\ell^-$ we can 
view the exponential multipliers as expansion coefficients,
we see that we can take as our frame for $\ell^-$ the 
matrices
\begin{equation*}
\begin{aligned}
X^- (x; \lambda) &= R^- + \mathbf{O} (|\lambda|^{-1/2}) \\
Y^- (x; \lambda) &= S^- + \mathbf{O} (1),
\end{aligned}
\end{equation*}
where the $\mathbf{O} (\cdot)$ terms are uniform for 
$x \in (-\infty, M]$, and we have observed that  
$\mu_j^- = \sqrt{-\lambda} \hat{\mu}_j^-$, for 
$j = 1, 2, \dots, 2n$. Likewise, we find that for 
$-\lambda > 0$ sufficiently large   
\begin{equation*}
\begin{aligned}
X^+_{\infty} (\lambda) &= R^+ + \mathbf{O} (|\lambda|^{-1/2}) \\
Y^+_{\infty} (\lambda) &= S^+ + \mathbf{O} (1).
\end{aligned}
\end{equation*}

Turning to $\tilde{W} (x; \lambda)$, we first observe that $S^- (\lambda)^{-1}$ can easily be identified, 
using the orthogonality of $R^-$; in particular, the $i$-th row of $S^- (\lambda)^{-1}$
is $\frac{1}{\mu_{n+i}^-} (r_i^-)^t$, which is $\mathbf{O} (|\lambda|^{-1/2})$. In 
this way, we see that 
\begin{equation*}
\begin{aligned}
X^- (x;\lambda) - i Y^- (x; \lambda) 
& = R^- + \mathbf{O} (|\lambda|^{-1/2}) -i S^- (\lambda) + \mathbf{O} (1) \\
&= - i S^- (\lambda) \Big{\{} i S^- (\lambda)^{-1} ( R^- + \mathbf{O} (1)) + I \Big{\}} \\
&= - i S^- (\lambda) (I + \mathbf{O} (|\lambda|^{-1/2})),
\end{aligned}
\end{equation*} 
and so 
\begin{equation*}
\begin{aligned}
(X^- (x;\lambda) - i Y^- (x;\lambda))^{-1} 
&= i (I + \mathbf{O} (|\lambda|^{-1/2}))^{-1} S^- (\lambda)^{-1} \\
&= i (I + \mathbf{O} (|\lambda|^{-1/2})) S^- (\lambda)^{-1},  
\end{aligned}
\end{equation*}
by Neumann approximation.
Likewise, 
\begin{equation*}
\begin{aligned}
X^- (x;\lambda) + i Y^- (x;\lambda) 
&= R^- + \mathbf{O} (|\lambda|^{-1/2}) + i S^- (\lambda) + \mathbf{O} (1) \\
&= (iI + (R^- + \mathbf{O} (1)) S^- (\lambda)^{-1} ) S^- (\lambda) \\
&= (iI + \mathbf{O} (|\lambda|^{-1/2})) S^- (\lambda). 
\end{aligned}
\end{equation*}
In this way, we see that 
\begin{equation*}
\begin{aligned}
(X^- (x;\lambda) &+ i Y^- (x;\lambda)) (X^- (x;\lambda) - i Y^- (x;\lambda))^{-1} \\
&= (iI + \mathbf{O} (|\lambda|^{-1/2})) S^- (\lambda) i (I + \mathbf{O} (|\lambda|^{-1/2})) S^- (\lambda)^{-1} \\
&= (iI + \mathbf{O} (|\lambda|^{-1/2})) (iI + \mathbf{O} (|\lambda|^{-1/2})) \\
&= -I + \mathbf{O} (|\lambda|^{-1/2}).
\end{aligned}
\end{equation*}

Proceeding similarly for $\mathbf{X}_{\infty} (\lambda)$, we have 
\begin{equation*}
(X^+_{\infty} (\lambda) + i Y^+_{\infty} (\lambda) (X^+_{\infty} (\lambda) - i Y^+_{\infty} (\lambda))^{-1}
= -I + \mathbf{O} (|\lambda|^{-1/2}),
\end{equation*} 
and so 
\begin{equation*}
\tilde{W} (x; \lambda) = -I + \mathbf{O} (|\lambda|^{-1/2})
\end{equation*}
uniformly in $x$. 
We see that for $\lambda_{\infty}$ sufficiently large the eigenvalues 
of $\tilde{W} (x; - \lambda_{\infty})$ are near $-1$
uniformly for $x \in (-\infty, M]$.

Turning to the behavior of $\tilde{W} (x; \lambda)$ as 
$x$ tends to $+\infty$ (i.e., for $x \ge M$), 
we recall from Section \ref{bound_section} that if $\lambda_{\infty}$
is large enough then $- \lambda_{\infty}$ will not be an 
eigenvalue of $H$. This means the evolving Lagrangian 
subspace $\ell^-$ cannot intersect the space of solutions 
asymptotically decaying as $x \to +\infty$, and so the 
frame $X^- (x;\lambda)$ must be comprised of solutions that 
grow as $x$ tends to $+\infty$. The construction of these 
growing solutions is almost identical to our construction of
the decaying solutions $\Phi_j^-$, and we'll be brief.

In this case, it's convenient to write 
\begin{equation*}
\mathbb{A} (\xi; \lambda) = \mathbb{A}_+ (\lambda) + \mathbb{E}_+ (\xi; \lambda),
\end{equation*} 
where 
\begin{equation*}
\mathbb{A}_+ (\lambda) = 
\begin{pmatrix}
0 & I \\
I - \frac{1}{\lambda} V_+ & 0
\end{pmatrix};
\quad 
\mathbb{E}_+ (\xi; \lambda) = 
\begin{pmatrix}
0 & 0 \\
\frac{1}{\lambda} (V_+ - V(\frac{\xi}{\sqrt{-\lambda}})) & 0
\end{pmatrix}.
\end{equation*}
The eigenvalues of $\mathbb{A}_+ (\lambda)$ can be 
expressed as 
\begin{equation*}
\begin{aligned}
\hat{\mu}_j^+ (\lambda) &= - \sqrt{1 - \frac{\nu^+_{m+1-j}}{\lambda}}
= \frac{1}{\sqrt{-\lambda}} \mu_j^+ \\
\hat{\mu}_{n+j}^+ (\lambda) &= \sqrt{1 - \frac{\nu_j^+}{\lambda}}
= \frac{1}{\sqrt{-\lambda}} \mu_{n+j}^+,
\end{aligned}
\end{equation*} 
for $j = 1,2, \dots, n$ (ordered, as usual, so that $j < k$ implies 
$\hat{\mu}_j^+ \le \hat{\mu}_k^+$). In order to select a solution growing with 
rate $\hat{\mu}_{n+j}^+$ (as $\xi \to +\infty$), we look for solutions 
of the form 
$\Phi (\xi; \lambda) = e^{\hat{\mu}_{n+j}^+ (\lambda) \xi} Z (\xi; \lambda)$,
for which $Z$ satisfies 
\begin{equation*}
Z' = (\mathbb{A}_+ (\lambda) - \hat{\mu}_{n+j}^+ (\lambda) I) Z 
+ \mathbb{E}_+ (\xi; \lambda) Z.
\end{equation*} 

Proceeding as with the frame of solutions that decay as $x \to -\infty$, 
we find that for $M$ sufficiently large (so that asymptotically decaying 
solutions become negligible), we can take as our frame for $\ell^-$ 
\begin{equation*}
\begin{aligned}
X^+ (x; \lambda) &= R^+ + \mathbf{O} (|\lambda|^{-1/2}) \\
Y^+ (x; \lambda) &= \tilde{S}^+ + \mathbf{O} (1),
\end{aligned}
\end{equation*}
where 
\begin{equation*}
{\tilde{S}}^+ = 
\begin{pmatrix}
\mu_{n+1}^+ r_n^+ & \mu_{n+2}^+ r_{n-1}^+ & \dots & \mu_{2n}^+ r_1^+
\end{pmatrix},
\end{equation*}
and the $\mathbf{O} (\cdot)$ terms are uniform for $x \in [M, \infty)$.
Proceeding now almost exactly as we did for the interval $(-\infty, M]$
we find that for $\lambda_{\infty}$ sufficiently large the eigenvalues 
of $\tilde{W} (x; - \lambda_{\infty})$ are near $-1$
uniformly for $x \in [M, \infty)$.

We summarize these considerations in a lemma. 

\begin{lemma} Let $V \in C(\mathbb{R})$ be a real-valued symmetric matrix, 
and suppose (A1) and (A2) hold. Then given any $\epsilon > 0$ there 
exists $\lambda_{\infty} > 0$ sufficiently large so that for all 
$x \in \mathbb{R}$ and for any eigenvalue 
$\omega (x; -\lambda_{\infty})$ of $\tilde{W} (x; -\lambda_{\infty})$
we have 
\begin{equation*}
|\omega (x; -\lambda_{\infty}) + 1| < \epsilon.
\end{equation*}
\end{lemma}

\begin{remark} We note that it would be insufficient to simply take 
$M = x_{\infty}$ in our argument. This is because our overall argument
is structured in such a way that we choose $\lambda_{\infty}$ first, 
and then choose $x_{\infty}$ sufficiently large, based on this value.
(This if for the bottom shelf argument.) But $\lambda_{\infty}$ must 
be chosen based on $M$, so $M$ should not depend on the value of 
$x_{\infty}$.
\end{remark}

We now make the following claim.

\begin{lemma} \label{left_lemma} 
Let $V \in C(\mathbb{R})$ be a real-valued symmetric matrix, and suppose (A1)
and (A2) hold. Then given any $M > 0$ there exists $\lambda_{\infty} > 0$ 
sufficiently large so that  
\begin{equation*}
\Mas (\ell^-, \ell_{\infty}^+; \Gamma_{\infty}) = 0,
\end{equation*}
for any $x_{\infty} > M$.
\end{lemma}

\begin{proof} We begin by observing that by taking $\lambda_{\infty}$ sufficiently
large, we can ensure that for all $x_{\infty} > M$ the eigenvalues of 
$\tilde{W} (x_{\infty}; -\lambda_{\infty})$
are all near $-1$. To make this precise, given any $\epsilon > 0$ we can take 
$\lambda_{\infty}$ sufficiently large so that the eigenvalues of 
$\tilde{W} (x_{\infty}; -\lambda_{\infty})$ are confined  to the arc 
$\mathcal{A}_{\epsilon} = \{e^{i \theta}: |\theta - \pi| < \epsilon\}$. 
Moreover, we know from Lemma \ref{monotonicity_lemma} that as $\lambda$ 
decreases toward $-\lambda_{\infty}$ the eigenvalues of $\tilde{W} (x_{\infty}; \lambda)$
will monotonically rotate in the counterclockwise direction, and so the 
eigenvalues of $\tilde{W} (x_{\infty}; -\lambda_{\infty})$ will in fact be confined to 
the arc $\mathcal{A}_{\epsilon}^+ = \{e^{i \theta}: -\epsilon < \theta - \pi < 0\}$.
(See Figure \ref{Ape}; we emphasize that none of the eigenvalues can cross $-1$, because such a crossing
would correspond with an eigenvalue of $H$, and we have assumed $\lambda_{\infty}$
is large enough so that there are no eigenvalues for $\lambda \le - \lambda_{\infty}$.)
Likewise, by the same monotonicity argument, we see that the eigenvalues of 
$\tilde{W}^- (-\lambda_{\infty})$ are also confined to $\mathcal{A}_{\epsilon}^+$.

\begin{figure}[ht] 
\begin{center}\includegraphics[%
  width=8cm,
  height=8cm]{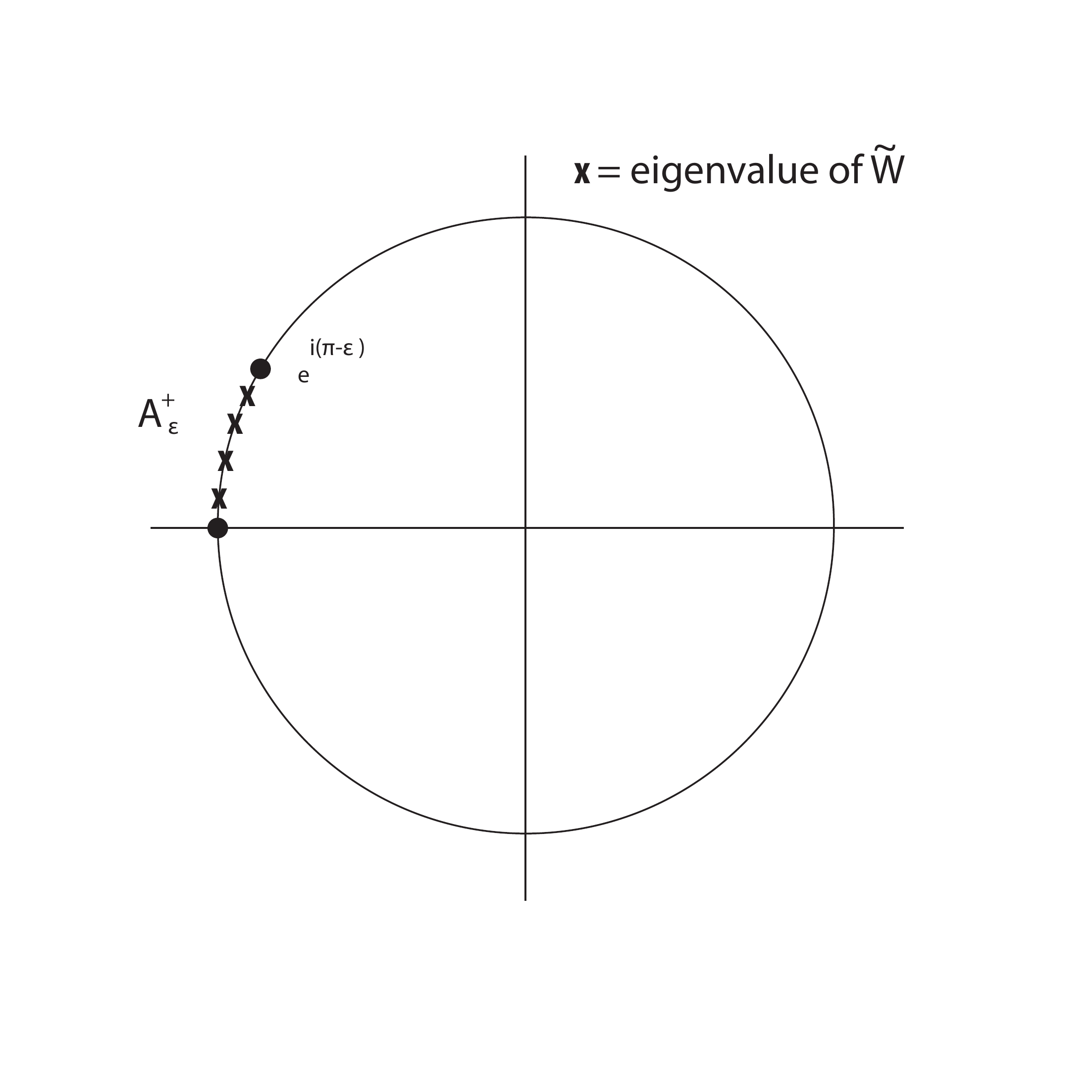}\end{center}
\caption{Eigenvalues confined to $A^+_{\epsilon}$. \label{Ape}}
\end{figure}

Turning now to the flow of eigenvalues as $x$ proceeds from $x_{\infty}$ to $-\infty$
(i.e., along the reverse direction of $\Gamma_{\infty}$), we note by uniformity 
of our large-$|\lambda|$
estimates that we can take $\lambda_{\infty}$ large enough so that the eigenvalues
of $\tilde{W}(x; -\lambda_{\infty})$ are confined to $\mathcal{A}_{\epsilon}$
(not necessarily $\mathcal{A}_{\epsilon}^+$) for all $x \in \mathbb{R}$. 
Combining these observations, we conclude that the eigenvalues of 
$\tilde{W}(x; -\lambda_{\infty})$ must begin and end in $\mathcal{A}_{\epsilon}^+$,
without completing a loop of $S^1$, and consequently the Maslov index along the 
entirety of $\Gamma_{\infty}$ must be 0.
\end{proof}

\subsection{Proof of Theorem \ref{main_theorem}} \label{proof_section}

Let $\Gamma$ denote the contour obtained by proceeding counterclockwise
along the paths $\Gamma_0$, $\Gamma_+$, $\Gamma_{\infty}$, $\Gamma_-$.
By the catenation property of the Maslov index, we have 
\begin{equation*}
\Mas (\ell^-, \ell^+_{\infty}; \Gamma) =
\Mas (\ell^-, \ell^+_{\infty}; \Gamma_0) + \Mas (\ell^-, \ell^+_{\infty}; \Gamma_+)
+ \Mas (\ell^-, \ell^+_{\infty}; \Gamma_{\infty}) + \Mas (\ell^-, \ell^+_{\infty}; \Gamma_-).
\end{equation*}
Moreover, by the homotopy property, and by noting that $\Gamma$ is homotopic to 
an arbitrarily small cycle attached to any point of $\Gamma$, we can conclude 
that $\Mas (\ell^-, \ell^+_{\infty}; \Gamma) = 0$. Since 
$\Mas (\ell^-, \ell^+_{\infty}; \Gamma_{\infty}) = 0$, and 
$\Mas (\ell^-, \ell^+_{\infty}; \Gamma_+) = \Mor (H)$, it follows immediately 
that 
\begin{equation} \label{loop}
\Mas (\ell^-, \ell^+_{\infty}; \Gamma_0) + \Mas (\ell^-, \ell^+_{\infty}; \Gamma_{-}) + \Mor (H)
= 0.
\end{equation} 

We will complete the proof with the following claim.

\begin{claim} Under the assumptions of Theorem \ref{main_theorem},
\begin{equation*}
\Mas (\ell^-, \ell^+_{\infty}; \Gamma_0) + \Mas (\ell^-, \ell^+_{\infty}; \Gamma_{-})
= \Mas(\ell^-, \ell^+_{\mathbf{R}}; \bar{\Gamma}_0).
\end{equation*}
\end{claim}

\begin{proof} First, consider the case 
\begin{equation*}
\kappa = \dim (\ell^-_{\mathbf{R}} (0) \cap \ell^+_{\mathbf{R}} (0)) = 0,
\end{equation*}
where we introduce the notation $\kappa$ for notational convenience.
In this case, we know from Lemma \ref{bottom_shelf_case1} that for 
$x_{\infty}$ sufficiently large we will have
$\Mas (\ell^-, \ell^+_{\infty}; \Gamma_-)  = 0$. It remains to show that 
\begin{equation*}
\Mas (\ell^-, \ell^+_{\infty}; \Gamma_0)
= \Mas(\ell^-, \ell^+_{\mathbf{R}}; \bar{\Gamma}_0).
\end{equation*}

As usual, let $\tilde{W} (x; \lambda)$ denote the unitary matrix 
(\ref{tildeWmain}) (which we recall depends on $x_{\infty}$), 
and let $\tilde{\mathcal{W}} (x; \lambda)$ denote the unitary 
matrix 
\begin{equation}
\begin{aligned}
\tilde{\mathcal{W}} (x; \lambda) 
&= - (X^- (x; \lambda) + i Y^- (x; \lambda)) (X^- (x; \lambda) - i Y^- (x; \lambda))^{-1} \\
&\times (R^+ - i S^+ (\lambda)) 
(R^+ + i S^+ (\lambda))^{-1}.
\end{aligned}
\end{equation}  
I.e., $\tilde{W} (x; \lambda)$ is the unitary matrix used in the
calculation of $\Mas (\ell^-, \ell^+_{\infty}; \Gamma_0)$ and 
$\tilde{\mathcal{W}} (x; \lambda)$ is the unitary matrix used
in the calculation of $\Mas(\ell^-, \ell^+_{\mathbf{R}}; \bar{\Gamma}_0)$. 
Likewise, set 
\begin{equation*}
\begin{aligned}
\tilde{W}^- (\lambda) &= \lim_{x \to -\infty} \tilde{W} (x; \lambda) \\
\tilde{\mathcal{W}}^- (\lambda) &= \lim_{x \to -\infty} \tilde{\mathcal{W}} (x; \lambda), 
\end{aligned}
\end{equation*}
both of which are well defined. (Notice that while the matrix 
$\tilde{\mathcal{W}} (x; \lambda)$ has not previously appeared, the other matrices 
here, including $\tilde{\mathcal{W}}^- (\lambda)$, are the same as 
before.)

By taking $x_{\infty}$ sufficiently large we can ensure that 
the spectrum of $\tilde{W}^- (0)$ is arbitrarily close to the spectrum 
of $\tilde{\mathcal{W}}^- (0)$ in the following sense: given 
any $\epsilon > 0$ we can take $x_{\infty}$ sufficiently large so 
that for any $\omega \in \sigma (\tilde{W}^- (0))$ there exists
$\tilde{\omega} \in \sigma (\tilde{\mathcal{W}}^- (0))$ so 
that $|\omega - \tilde{\omega}| < \epsilon$. 

Turning to the other end of our contours, we first take the case 
$\nu_{\min} > 0$ so that $\lambda = 0$ is not embedded in 
essential spectrum. In this case, $\tilde{W} (x_{\infty};0)$ 
will have $-1$ as an eigenvalue if and only if $\lambda = 0$ 
is an eigenvalue of $H$, and the multiplicity of $-1$ as an 
eigenvalue of $\tilde{W} (x_{\infty};0)$ will correspond with 
the geometric multiplicity of $\lambda = 0$ as an eigenvalue of
$H$. For $\tilde{\mathcal{W}} (x; 0)$ set 
\begin{equation*}
\tilde{\mathcal{W}}^+ (0) = \lim_{x \to \infty} \tilde{\mathcal{W}} (x; 0),
\end{equation*}
which is well defined by our construction in the appendix. 
As with $\tilde{W} (x_{\infty};0)$, $\tilde{\mathcal{W}}^+ (0)$ 
will have $-1$ as an eigenvalue if and only if $\lambda = 0$ 
is an eigenvalue of $H$, and the multiplicity of $-1$ as an 
eigenvalue of $\tilde{\mathcal{W}}^+ (0)$ will correspond with 
the geometric multiplicity of $\lambda = 0$ as an eigenvalue of
$H$. By choosing $x_{\infty}$ sufficiently large, we can ensure 
that the eigenvalues of $\tilde{W} (x_{\infty};0)$ are arbitrarily 
close to the eigenvalues of $\tilde{\mathcal{W}}^+ (0)$. I.e., 
$-1$ repeats as an eigenvalue the same number of times for these
two matrices, and the eigenvalues aside from $-1$ can be made
arbitrarily close. 

We see that the path of matrices $\tilde{W} (x;0)$, as $x$ runs
from $-\infty$ to $x_{\infty}$ can be viewed as a small perturbation 
from the path of matrices $\tilde{\mathcal{W}} (x;0)$, as 
$x$ runs from $-\infty$ to $+\infty$. In order to clarify this, 
we recall that by using the change of variables (\ref{change}) we
can specify our path of Lagrangian subspaces on the compact interval
$[-1, 1]$. Likewise, the interval $(-\infty, x_{\infty}]$ compactifies
to $[-1, (e^{x_{\infty}} - 1)/(e^{x_{\infty}} + 1)]$. For this 
latter interval, we can make the further change of variables 
\begin{equation*}
\xi = \frac{2}{1+r_{\infty}} \tau + \frac{1 - r_{\infty}}{1 + r_{\infty}},
\end{equation*}
where $r_{\infty} = (e^{x_{\infty}} - 1)/(e^{x_{\infty}} + 1)$, so 
that $\tilde{W} (x;0)$ and $\tilde{\mathcal{W}} (x; 0)$ can both be specfied on 
the interval $[-1, 1]$. Finally, we see that 
\begin{equation*}
|\xi - \tau| = (1+\tau) \frac{1-r_{\infty}}{1+r_{\infty}},
\end{equation*}
so by choosing $x_{\infty}$ sufficiently large (and hence $r_{\infty}$
sufficiently close to $1$), we can take the values of $\xi$ and $\tau$
as close as we like. By uniform continuity the eigenvalues of the 
adjusted path will be arbitrarily close to those of the original
path. 
Since the endstates associated with these paths are arbitrarily close, 
and since the eigenvalues of
one path end at $-1$ if and only if the eigenvalues of the other path 
do, the homotopy invariance argument in \cite{HLS} can be employed to show 
that the spectral flow must be the same along each of these paths, 
and this establishes the claim.

In the event that $\nu_{\min} = 0$ so that $\lambda = 0$ is 
embedded in essential spectrum, it may be the case that $\tilde{W} (x_{\infty};0)$
has $-1$ as an eigenvalue even if $\lambda = 0$ is not an eigenvalue.
More generally, the multiplicity of $-1$ as an eigenvalue of 
$\tilde{W} (x_{\infty};0)$ may not correspond with the geometric 
multiplicity of $\lambda = 0$ as an eigenvalue of $H$. Rather, 
in such cases the multiplicity of $-1$ as an eigenvalue of 
$\tilde{W} (x_{\infty};0)$ will correspond with the dimension of the 
intersection of the space of solutions
that are obtained obtained as $\lambda \to 0^-$ limits of solutions
that decay at $-\infty$ and the space of solutions that are obtained 
as $\lambda \to 0^-$ limits of solutions that decay at $+\infty$. 
(Here, we are keeping in mind that as $\lambda \to 0^-$ if a decaying 
solution ceases to decay then there will be a corresponding growing 
solution that ceases to grow.) Once again, $\tilde{\mathcal{W}}^+ (0)$
will have $-1$ as an eigenvalue if and only if $\tilde{W} (x_{\infty};0)$
does, and we will be able to apply the same argument as discussed above
to establish the claim.

We now turn to the case 
\begin{equation*}
\kappa = \dim (\ell^-_{\mathbf{R}} (0) \cap \ell^+_{\mathbf{R}} (0)) \ne 0,
\end{equation*} 
and as with the case $\kappa = 0$ we begin by assuming $\nu_{\min} > 0$.
The matrix $\tilde{\mathcal{W}}^- (0)$ will have $-1$ as an eigenvalue
with multiplicity $\kappa$. By monotonicity in $\lambda$, and Lemma 
\ref{bottom_shelf_case1} we know that for any $\lambda < 0$ the eigenvalues
will have rotated away from $-1$ in the counterclockwise direction. In 
particular, given any $\epsilon > 0$ we can find $\lambda_0 > 0$ sufficiently
small so that the eigenvalues of $\tilde{\mathcal{W}}^- (-\lambda_0)$ are on the arc 
\begin{equation*}
A_{\epsilon}^- = \{e^{i \theta}: \pi < \theta < \pi + \epsilon\},
\end{equation*} 
while no other eigenvalues of $\tilde{\mathcal{W}}^- (0)$ are on 
the arc $A_{\epsilon}^-$. 

Recalling that $\tilde{W}^- (0)$ can be viewed as a small perturbation
of $\tilde{\mathcal{W}}^- (0)$, we see that for $x_{\infty}$ sufficiently
large there will be a cluster of $\kappa$ eigenvalues of $\tilde{W}^- (0)$ near 
$-1$, on an arc $A_{\tilde{\epsilon}}$, where $\tilde{\epsilon}$ can be
made as small as we like by our choice of $x_{\infty}$. Moreover, by monotonicity
in $\lambda$, we can choose $\lambda_0 > 0$ sufficiently small (perhaps smaller
than the previous choice) so that the corresponding eigenvalues of 
$\tilde{W}^- (-\lambda_0)$ are confined to the arc $A_{\tilde{\epsilon}}^-$.
See Figure \ref{eigs_close}.

\begin{figure}[ht] 
\begin{center}\includegraphics[%
  width=8cm,
  height=8cm]{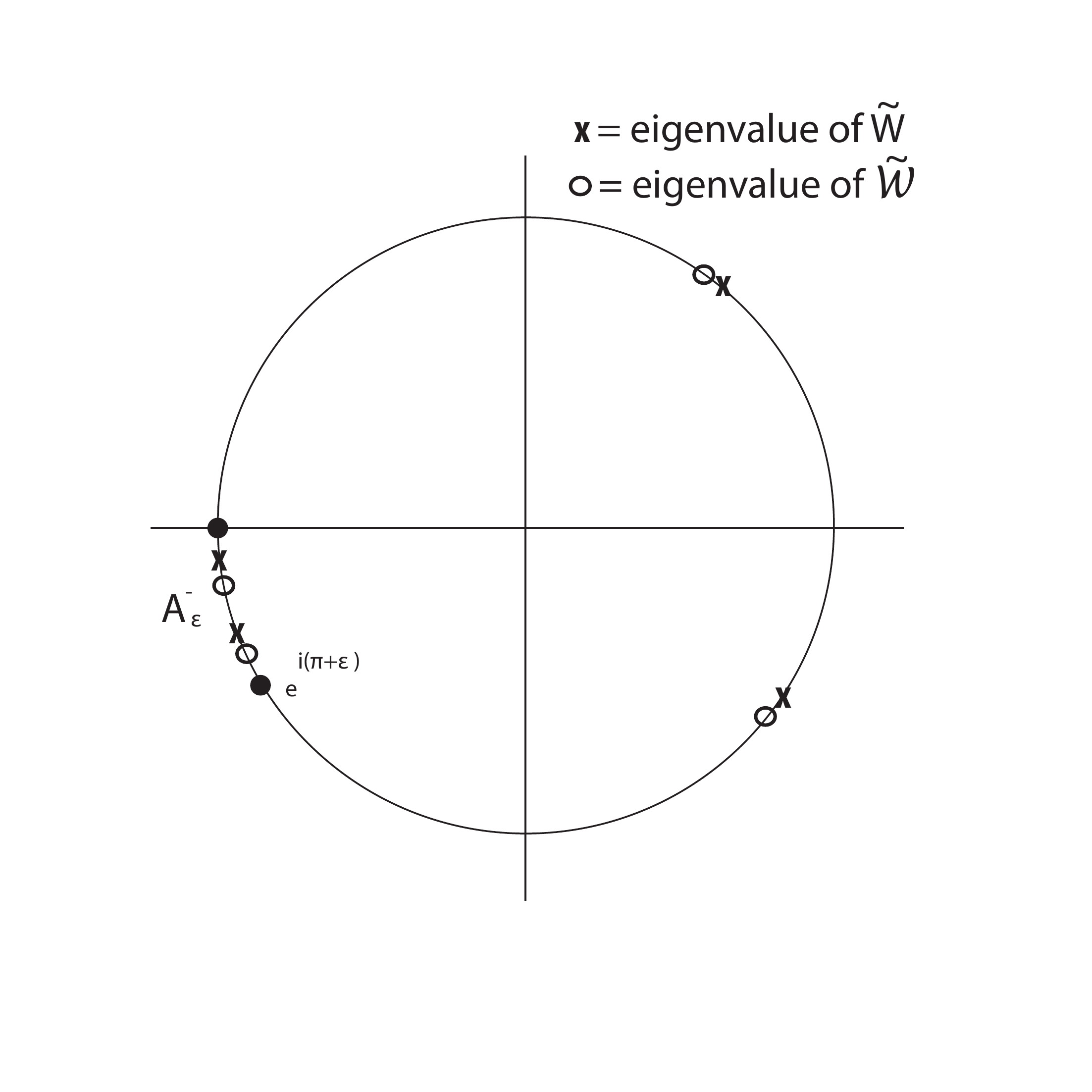}\end{center}
\caption{The eigenvalues of $\tilde{W}^- (-\lambda_0)$ and $\tilde{\mathcal{W}} (-\lambda_0)$. \label{eigs_close}}
\end{figure}

At this point, we can consider the spectral flow of the family of
matrices $\tilde{W} (x;\lambda)$ along the path obtained by first 
fixing $x = -\infty$ and letting $\lambda$ run from $-\lambda_0$
to 0, and then fixing $\lambda = 0$ and letting $x$ run from 
$-\infty$ to $x_{\infty}$. We denote this path $\Gamma_1$.  
Correspondingly, we can consider the spectral flow of the family of 
matrices $\tilde{\mathcal{W}} (x;\lambda)$ along the path obtained 
by first fixing $x = -\infty$ and letting $\lambda$ run from 
$-\lambda_0$ to 0, and then fixing $\lambda = 0$ 
and letting $x$ run from $-\infty$ to $+\infty$. We denote this 
path $\Gamma_2$. 

We are now in almost precisely the same case as when $\kappa = 0$, 
and again we can use an argument similar to the homotopy argument of 
\cite{HLS} to verify that these two spectral flows will give equivalent 
values. This means, of course, that the Maslov indices along 
these paths for the respective pairs $(\ell^-, \ell^+_{\infty})$
(for $\Gamma_1$)
and $(\ell^-, \ell^+_{\mathbf{R}})$ (for $\Gamma_2$) will be equivalent. 
I.e., 
\begin{equation*}
\Mas (\ell^-, \ell^+_{\mathbf{R}}; \Gamma_1)
= \Mas (\ell^-, \ell^+_{\infty}; \Gamma_2). 
\end{equation*}

However, by Lemma \ref{bottom_shelf_case1} we know that the pair 
$(\ell^-, \ell^+_{\mathbf{R}})$ has no intersections along
$\Gamma_-$ (except at $\lambda = 0$). By monotonicity, as 
$\lambda \to 0^-$ the eigenvalues of $\tilde{\mathcal{W}}^- (\lambda)$
will rotate in the clockwise direction, so those rotating to $-1$
will not increment the Maslov index. We conclude that 
\begin{equation*}
\Mas (\ell^-, \ell^+_{\mathbf{R}}; \Gamma_2) = 
\Mas (\ell^-, \ell^+_{\mathbf{R}}; \bar{\Gamma}_0),
\end{equation*}
where we recall from the introduction that $\bar{\Gamma}_0$ is the contour
obtained by fixing $\lambda = 0$ and letting $x$ run from $-\infty$
to $+\infty$.
Likewise, according to Lemma \ref{bottom_shelf_case1} we can take $x_{\infty}$ 
sufficiently large so that $\tilde{W}^- (\lambda)$ does not have $-1$ as
an eigenvalue for any $\lambda \in [-\lambda_{\infty}, - \lambda_0]$. This
implies that 
\begin{equation*}
\Mas (\ell^-, \ell^+_{\infty}; \Gamma_1) = 
\Mas (\ell^-, \ell^+_{\infty}; \Gamma_-) + \Mas (\ell^-, \ell^+_{\infty}; \Gamma_0).
\end{equation*} 
Combining, we find 
\begin{equation*}
\Mas (\ell^-, \ell^+_{\mathbf{R}}; \bar{\Gamma}_0) = 
\Mas (\ell^-, \ell^+_{\infty}; \Gamma_-) + \Mas (\ell^-, \ell^+_{\infty}; \Gamma_0),
\end{equation*} 
which is the claim. The case $\nu_{\min} = 0$ for $\kappa \ne 0$ follows 
similarly as for $\kappa = 0$.
\end{proof}

Upon combining the claim with (\ref{loop}), we obtain Theorem \ref{main_theorem}.

\section{Equations with Constant Convection} \label{convection_section}

For a traveling wave solution $\bar{u} (x - st)$ to the Allen-Cahn equation 
\begin{equation} \label{AC}
u_t + F'(u) = u_{xx},
\end{equation}
it's convenient to switch to a shifted coordinate frame in which the wave 
is a stationary solution $\bar{u} (x)$ for the equation 
\begin{equation} \label{SAC}
u_t - su_x + F'(u) = u_{xx}.
\end{equation}
In this case, linearization about the wave leads to an eigenvalue problem 
\begin{equation} \label{Es}
H_s y := - y'' + s y' + V(x) y = \lambda y,
\end{equation}
where $V(x) = F'' (\bar{u} (x))$. 
Our goal in this section is to show that our development for 
(\ref{main}) can be extended to the case (\ref{Es}) in a 
straightforward manner. For this discussion, which is adapted from 
\cite{BJ1995}, we take any real number $s \ne 0$, and we continue 
to let assumptions (A1) and (A2) hold.

The main issues we need to address are as follows: (1) we need to show
that the point spectrum for $H_s$ is real-valued; (2) we need to show 
that the $n$-dimensional subspaces associated with $H_s$ are Lagrangian; 
and (3) we need to show that the eigenvalues of the associated 
unitary matrix $\tilde{W} (x; \lambda)$ rotate monotonically as 
$\lambda$ increases (or decreases). Once these items have been verified, 
the remainder of our analysis carries over directly to the case
$s \ne 0$.

\subsection{Essential Spectrum} \label{essential_spectrum_subsection}

As for the case $s = 0$ the essential spectrum for $s \ne 0$ can be 
identified from the asymptotic equations 
\begin{equation} \label{Asym_Es}
- y'' + s y' + V_{\pm} y = \lambda y.
\end{equation}
Precisely, the essential spectrum will correspond with 
values of $\lambda$ for which (\ref{Asym_Es}) admits a solution of the form 
$y(x) = e^{ikx} r$ for some constant non-zero vector $r \in \mathbb{C}^n$. Upon 
substitution of this ansatz into (\ref{Es}) we obtain the relations 
\begin{equation*}
(k^2 I + iskI + V_{\pm}) r = \lambda r.
\end{equation*} 
We take a $\mathbb{C}^n$ inner product with $r$ to see that 
\begin{equation*}
\lambda (k) |r|^2 = (V_{\pm} r, r)_{\mathbb{C}^n} + (i s k + k^2) |r|^2, 
\end{equation*}
or equivalently 
\begin{equation*}
\lambda (k) = \frac{(V_{\pm} r, r)_{\mathbb{C}^n}}{|r|^2} + i s k + k^2.
\end{equation*}
We conclude that the essential spectrum is confined on and to the right of 
parabolas opening into the real complex half-plane, described by the 
relations 
\begin{equation*}
\text{Re } \lambda = \frac{(V_{\pm} r, r)_{\mathbb{C}^n}}{|r|^2} + \frac{1}{s^2} (\text{Im }\lambda)^2.
\end{equation*}
For notational convenience, we denote by $\Omega$ this region in $\mathbb{C}$
on or two the right of these parabolas.

\subsection{In $\mathbb{C} \backslash \Omega$ the Point Spectrum of $H_s$ is Real-Valued} \label{real_section}

For any $\lambda \in \mathbb{C} \backslash \Omega$, we can look for ODE solutions with asymptotic 
behavior $y(x) = e^{\mu x} r$. Upon substitution into (\ref{Asym_Es}) we obtain 
the eigenvalue problem 
\begin{equation*}
(-\mu^2 + s \mu + V_{\pm} - \lambda) r = 0.
\end{equation*} 
As in Section \ref{ODEsection} we denote the eigenvalues of $V_{\pm}$
by $\{\nu_j^{\pm}\}_{j=1}^n$, with associated eigenvectors $\{r_j^{\pm}\}_{j=1}^n$.
We see that the possible growth/decay rates $\mu$ will 
satisfy
\begin{equation*}
\mu^2 -s \mu + \lambda = \nu_j^{\pm}
\implies
\mu = \frac{s \pm \sqrt{s^2 - 4(\lambda - \nu_j^{\pm})}}{2}.
\end{equation*}
We label the $2n$ growth/decay rates similarly as in 
Section \ref{ODEsection}
\begin{equation*}
\begin{aligned}
\mu_j^{\pm} (\lambda) &= \frac{s - \sqrt{s^2 - 4(\lambda - \nu_{n+1-j}^{\pm})}}{2} \\
\mu_{n+j}^{\pm} (\lambda) &= \frac{s + \sqrt{s^2 - 4(\lambda - \nu_{j}^{\pm})}}{2},
\end{aligned}
\end{equation*}
for $j = 1, 2, \dots, n$. 

Now, suppose $\lambda \in \mathbb{C} \backslash \Omega$ is is an eigenvalue for 
$H_s$. For this fixed value, we can obtain asymptotic ODE estimates on solutions of 
(\ref{Es}) with precisely the same form as those described in Lemma \ref{ODElemma}
(keeping in mind that the specifications of $\{\mu_j^{\pm}\}_{j=1}^{2n}$ are 
different). 
Letting $\psi (x; \lambda)$ denote the eigenfunction associated with $\lambda$,
we conclude that $\psi(x; \lambda)$ can be expressed both as a linear combination 
of the solutions that decay as $x \to -\infty$ (i.e., those associated with rates
$\{\mu_{n+j}^-\}_{j=1}^n$) and as a linear combination of the solutions that decay 
as $x \to +\infty$ (i.e., those associated with rates $\{\mu_{j}^+\}_{j=1}^n$).   

Keeping in mind that we are in the case $s \ne 0$, we make the change of variable
$\phi(x) = e^{-\frac{s}{2} x} y(x)$, for which a direct calculation yields 
\begin{equation*}
\mathcal{H}_s \phi := e^{-\frac{s}{2} x} H_s e^{\frac{s}{2} x} \phi 
= - \phi'' + (\frac{s^2}{4} + V(x)) \phi = \lambda \phi. 
\end{equation*}
Moreover, if $y(x)$ is a solution of $H_s y = \lambda y$ that decays with 
rate $\mu_{n+j}^- (\lambda)$ as $x \to -\infty$ then the corresponding 
$\phi(x)$ will decay as $x \to -\infty$ with rate 
\begin{equation}
-\frac{s}{2} + \frac{s + \sqrt{s^2 - 4(\lambda - \nu_{j}^{\pm})}}{2}
= 
\frac{1}{2} \sqrt{s^2 - 4(\lambda - \nu_{j}^{\pm})} > 0,
\end{equation}
and likewise if $y(x)$ is a solution of $H_s y = \lambda y$ that decays with 
rate $\mu_{j}^+ (\lambda)$ as $x \to +\infty$ then the corresponding 
$\phi(x)$ will decay as $x \to +\infty$ with rate 
\begin{equation}
-\frac{s}{2} + \frac{s - \sqrt{s^2 - 4(\lambda - \nu_{j}^{\pm})}}{2}
= 
- \frac{1}{2} \sqrt{s^2 - 4(\lambda - \nu_{j}^{\pm})} < 0.
\end{equation}

In this way we see that $\varphi(x; \lambda) = e^{-\frac{s}{2} x} \psi(x; \lambda)$
is an eigenfunction for $\mathcal{H}_s$, associated with the eigenvalue 
$\lambda$. But $\mathcal{H}_s$ is self-adjoint, and so its spectrum is 
confined to $\mathbb{R}$. We conclude that $\lambda \in \mathbb{R}$.

Finally, we observe that although the real value $\lambda = \nu_{\min}$ is 
embedded in the essential spectrum, it is already in $\mathbb{R}$. In this 
way, we conclude that any eigenvalue $\lambda$ of $H_s$ with 
$\text{Re }\lambda \le \nu_{\min}$ must be real-valued.

\subsection{Bound on the Point Spectrum of $H_s$} \label{point_spectrum_subsection}

Suppose $\lambda \in \mathbb{R}$ is an eigenvalue of $H_s$ with associated 
eigenvector $\psi (x; \lambda)$. Taking an $L^2 (\mathbb{R})$ inner product
of $H_s \psi = \lambda \psi$ with $\psi$ we obtain the relation 
\begin{equation*}
\|\psi'\|^2 + s \langle \psi', \psi \rangle + \langle V \psi, \psi \rangle
= \lambda \|\psi\|^2. 
\end{equation*}
We see that 
\begin{equation*}
\begin{aligned}
\lambda \|\psi\|^2 &\ge \|\psi'\|^2 - |s| \|\psi'\| \|\psi\| + \langle V \psi, \psi \rangle \\
&\ge \|\psi'\|^2 - |s| (\frac{\epsilon}{2} \|\psi'\|^2 + \frac{1}{2\epsilon} \|\psi\|^2) 
- \|V\|_{\infty} \|\psi\|^2 \\
&\ge 
- (\frac{1}{2\epsilon} + \|V\|_{\infty}) \|\psi\|^2,
\end{aligned}
\end{equation*}
from which we conclude that $\lambda$ is bounded below. (In this calculation, 
$\epsilon > 0$ has been taken sufficiently small.)

\subsection{The Spaces $\ell^- (x; \lambda)$ and $\ell^-_{\mathbf{R}} (\lambda)$ are Lagrangian}
\label{lagrangian_section}

Since $\sigma_p (H) \subset \mathbb{R}$, we can focus on $\lambda \in \mathbb{R}$, 
in which case the growth/decay rates $\{\mu_j^{\pm}\}_{j=1}^{2n}$ remain ordered 
as $\lambda$ varies. In light of this, the estimates of Lemma \ref{ODElemma} remain 
valid precisely as stated, with our revised definitions of these rates. The Lagrangian
property for $\mathbf{R}^- = {R^- \choose S^-}$ can be verified precisely as before, 
but for $\mathbf{X}^- (x; \lambda) = {X^- (x; \lambda) \choose Y^- (x; \lambda)}$
the calculation changes slightly. For this, take $\lambda \le \nu_{\min}$ and 
temporarily set 
\begin{equation*}
A(x; \lambda) := X^- (x; \lambda)^t Y^- (x; \lambda) - Y^- (x; \lambda)^t X^- (x; \lambda),
\end{equation*}   
and compute (letting prime denote differentiation with respect to $x$)
\begin{equation*}
\begin{aligned}
A'(x; \lambda) &= X^{- \,\prime} (x; \lambda)^t Y^- (x; \lambda) + X^- (x; \lambda)^t Y^{- \,\prime} (x; \lambda) \\
& - Y^{- \,\prime} (x; \lambda)^t X^- (x; \lambda) - Y^- (x; \lambda)^t X^{- \,\prime} (x; \lambda).
\end{aligned}
\end{equation*}
Using the relations 
\begin{equation} \label{more_useful_relations}
X^{- \,\prime} (x; \lambda) = Y^- (x; \lambda); 
\quad Y^{- \,\prime} (x; \lambda) = (V(x) - \lambda I) X^- (x; \lambda) + s Y^- (x; \lambda),
\end{equation}
we find that 
\begin{equation*}
A' (x; \lambda) = s A (x; \lambda).
\end{equation*}
It follows immediately that $e^{-s x} A (x; \lambda) = c$ for some constant $c$. But that rates of decay 
associated with $A(x; \lambda)$ have the form 
\begin{equation*}
\mu_{n+j}^- (\lambda) + \mu_{n+k}^- (\lambda) 
= s + \frac{1}{2} \sqrt{s^2 - 4 (\lambda - \nu_j^-)} + \frac{1}{2} \sqrt{s^2 - 4 (\lambda - \nu_k^-)},
\end{equation*}
from which we see that the exponents associated with $e^{-s x} A (x; \lambda)$ take the form 
\begin{equation*}
\frac{1}{2} \sqrt{s^2 - 4 (\lambda - \nu_j^-)} + \frac{1}{2} \sqrt{s^2 - 4 (\lambda - \nu_k^-)} > 0.
\end{equation*}
It is now clear that by taking $x \to -\infty$ we can conclude that $c = 0$. We conclude that 
$A(x; \lambda) = 0$ for all $x \in \mathbb{R}$, and it follows that $\mathbf{X}^- (x; \lambda)$
is the frame for a Lagrangian subspace (see Proposition 2.1 of \cite{HLS}).

\subsection{Monotoncity} \label{monotonicity_subsection}

In this case, according to Lemma 4.2 in \cite{HLS} monotonicity of $\tilde{W} (x; \lambda)$
(in $\lambda$) will be determined by the matrices 
\begin{equation} \label{submon}
X^- (x; \lambda) \partial_{\lambda} Y^- (x; \lambda) 
- Y^- (x; \lambda) \partial_{\lambda} X^- (x; \lambda) 
\end{equation} 
and 
\begin{equation} \label{submon2}
X^+_{\infty} (\lambda) \partial_{\lambda} Y^+_{\infty} (\lambda) 
- Y^+_{\infty} (\lambda) \partial_{\lambda} X^+_{\infty} (\lambda). 
\end{equation} 
(On the bottom shelf, (\ref{submon}) will be replaced by 
$(R^-)^t \partial_{\lambda} S^- (\lambda) - S^- (\lambda)^t \partial_{\lambda} R^-$.)

Let's temporarily set 
\begin{equation*}
B (x; \lambda) := X^- (x; \lambda) \partial_{\lambda} Y^- (x; \lambda) 
- Y^- (x; \lambda) \partial_{\lambda} X^- (x; \lambda), 
\end{equation*}
and compute (letting prime denote differentiation with respect to $x$)
\begin{equation*}
\begin{aligned}
B' (x; \lambda) &:= X^{- \, \prime} (x; \lambda) \partial_{\lambda} Y^- (x; \lambda) 
+ X^{-} (x; \lambda) \partial_{\lambda} Y^{- \, \prime} (x; \lambda) \\
& - Y^{- \, \prime} (x; \lambda) \partial_{\lambda} X^- (x; \lambda)
- Y^{-} (x; \lambda) \partial_{\lambda} X^{- \, \prime} (x; \lambda) \\
&= - X^- (x; \lambda)^t X^- (x; \lambda) + s B (x; \lambda),   
\end{aligned}
\end{equation*}
where we have used (\ref{more_useful_relations}) to get this final relation. 
Integrating this last expression, we find that 
\begin{equation*}
B(x; \lambda) = - \int_{-\infty}^x e^{s (x - y)} X^- (y; \lambda)^t X^- (y; \lambda) dy,
\end{equation*}
from which we conclude that $B(x; \lambda)$ is negative definite. We can proceed
similarly to verify that (\ref{submon2}) is also positive definite, and the matrix
associated with the bottom shelf can be analyzed as in the case $s = 0$.

\section{Applications} \label{applications_section}

In this section, we discuss three illustrative examples that we hope
will clarify the analysis. For the first two, which are adapted from 
\cite{CDB09}, we will be able to 
carry out explicit calculations for a range of values of $\lambda$. 
The third example, adapted from \cite{HK1}, will employ Theorem \ref{main_theorem}
more directly, in that we will determine that a certain operator
has no negative eigenvalues by computing only the principal 
Maslov index. 

\subsection{Example 1.} \label{example1_section}

We consider the Allen-Cahn equation 
\begin{equation*}
u_t = u_{xx} - u + u^2,
\end{equation*}
which is known to have a pulse-type stationary solution
\begin{equation*}
\bar{u} (x) = \frac{3}{2} \sech^2 (\frac{x}{2}).
\end{equation*}
(See \cite{CDB09}.) Linearizing about $\bar{u} (x)$ 
we obtain the eigenvalue problem
\begin{equation*}
-y'' + (1 - 2 \bar{u} (x)) y = \lambda y, 
\end{equation*}
which has the form (\ref{main}) with $n=1$ 
and $V(x) = 1-2\bar{u}(x)$ (for which (A1)-(A2)
are clearly satisfied). Setting $\Phi = {y \choose y'}$, 
we can express this equation as a first order system
$\Phi' = \mathbb{A} (x;\lambda) \Phi$, with 
\begin{equation} \label{esystem1}
\mathbb{A} (x;\lambda) =
\begin{pmatrix}
0 & 1 \\
1-2\bar{u}(x)-\lambda & 0
\end{pmatrix}; \quad
\mathbb{A}_{\pm} =  
\begin{pmatrix}
0 & 1 \\
1-\lambda & 0
\end{pmatrix}. 
\end{equation}

As observed in \cite{CDB09}, this equation can be solved
exactly for all $x \in \mathbb{R}$ and $\lambda < 1$
(in this case $\sigma_{ess} (H) \subset [1,\infty)$). 
In particular, if we set $s = \frac{x}{2}$, 
$\gamma = 2\sqrt{1-\lambda}$, and 
\begin{equation*}
H^{\pm} (s,\lambda) = \mp a_0 + a_1 \tanh s \mp \tanh^2 s + \tanh^3 s, 
\end{equation*}
with 
\begin{equation*}
a_0 = \frac{\gamma}{15} (4-\gamma^2); \quad 
a_1 = \frac{1}{5} (2\gamma^2-3); \quad
a_2 = -\gamma, 
\end{equation*}
then (\ref{esystem1}) has (up to multiplication by a constant) exactly one solution that decays as $x \to -\infty$,
\begin{equation*}
\Phi^- (x;\lambda) = e^{\gamma s}
\begin{pmatrix}
H^- (s, \lambda) \\
\frac{1}{2} (H_s^- (s,\lambda) + \gamma H^- (s,\lambda))
\end{pmatrix},
\end{equation*}
and 
exactly one solution that decays as $x \to +\infty$,
\begin{equation*}
\Phi^+ (x;\lambda) = e^{-\gamma s}
\begin{pmatrix}
H^+ (s, \lambda) \\
\frac{1}{2} (H_s^+ (s,\lambda) - \gamma H^+ (s,\lambda))
\end{pmatrix}.
\end{equation*} 

The target space can be obtained either from $\Phi^+ (x;\lambda)$
(by taking $x \to \infty$) or by working with $\mathbb{A}_+ (\lambda)$
directly (as discussed during our analysis), and in either case we
find that a frame for the target space is 
$\mathbf{R}^+ = {R^+ \choose S^+} = {1 \choose - \sqrt{1-\lambda}}$. 
Computing directly, we see that 
\begin{equation*}
(R^+ - iS^+ (\lambda)) (R^+ + i S^+ (\lambda))^{-1} 
= \frac{1+i\sqrt{1-\lambda}}{1-i\sqrt{1-\lambda}}.
\end{equation*}
Likewise, the evolving frame in this case can be taken to be
\begin{equation*}
\mathbf{X}^- (x; \lambda) = {X^- (x; \lambda) \choose Y^- (x, \lambda)}
= 
\begin{pmatrix}
H^- (s, \lambda) \\
\frac{1}{2} (H_s^- (s,\lambda) + \gamma H^- (s,\lambda))
\end{pmatrix}.
\end{equation*}
We set 
\begin{equation*}
\tilde{\mathcal{W}} (x; \lambda) = (X^- (x; \lambda) + i Y^- (x; \lambda)) (X^- (x; \lambda) - i Y^- (x; \lambda))^{-1}
(R^+ - i S^+ (\lambda)) (R^+ + i S^+ (\lambda))^{-1},
\end{equation*}
which in this case we can compute directly. The results of such a 
calculation, carried out in MATLAB, are depicted in 
Figure \ref{example1_figure}. 

\begin{remark} For the Maslov Box, we should properly use $\tilde{W} (x; \lambda)$
as defined in (\ref{tildeWmain}) for some sufficiently large $x_{\infty}$, but for
the purpose of graphical illustration (see Figure \ref{example1_figure}) there 
is essentially no difference between working with $\tilde{\mathcal{W}} (x; \lambda)$
and working with $\tilde{W} (x;\lambda)$ defined with $x_{\infty} = 3$.
\end{remark} 

Referring Figure \ref{example1_figure}, the curves comprise 
$x$-$\lambda$ pairs for which $\tilde{\mathcal{W}} (x; \lambda)$ has
$-1$ as an eigenvalue. The eigenvalues in this case are 
known to be $-\frac{5}{4}$, $0$, and $\frac{3}{4}$, and we
see that these are the locations of crossings along the top 
shelf, which for plotting purposes we've indicated at 
$x = 3$ for this example. We note particularly that the 
Principal Maslov Index is -1, because the path 
$\Gamma_0$ is only crossed once (the middle curve approaches
$\Gamma_0$ asymptotically, but this does not increment the
Maslov index).  

\begin{figure}[ht] 
\begin{center}\includegraphics[%
  width=10cm,
  height=8cm]{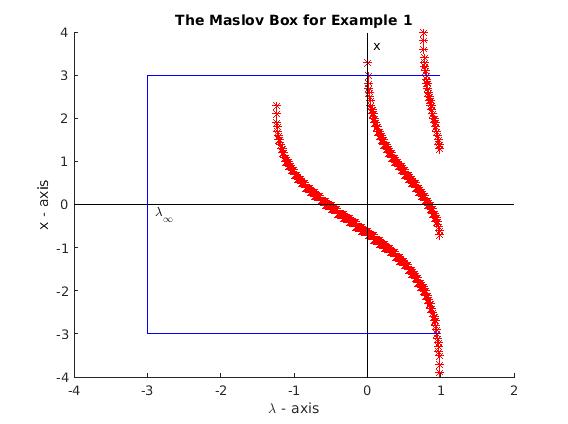}\end{center}
\caption{Figure for Example 1. \label{example1_figure}}
\end{figure}

\subsection{Example 2.} \label{example2_section}

We consider the Allen-Cahn system 
\begin{equation} \label{AC_sys}
\begin{aligned}
u_t &= u_{xx} - 4u + 6u^2 - c(u-v) \\
v_t &= v_{xx} - 4v + 6v^2 + c(u-v),
\end{aligned}
\end{equation}
where $c > - 2$, with also $c \ne 0$. System (\ref{AC_sys})
is known to have a stationary solution 
\begin{equation*}
\begin{aligned}
\bar{u} (x) &= \sech^2 x \\
\bar{v} (x) &= \sech^2 x
\end{aligned}
\end{equation*}
(see \cite{CDB09}). Linearizing about this vector solution, 
we obtain the eigenvalue system 
\begin{equation} \label{eg2E}
\begin{aligned}
- \phi'' + (4 - 12\bar{u} (x) + c) \phi - c \psi &= \lambda \phi \\
- \psi'' - c \phi + (4 - 12 \bar{v} (x) + c) \psi &= \lambda \psi,
\end{aligned}
\end{equation}
which can be expressed in form (\ref{main}) with 
$y = {\phi \choose \psi}$ and 
\begin{equation*}
V (x) = 
\begin{pmatrix}
(4 - 12\bar{u} (x) + c) & -c \\
-c & (4 - 12 \bar{v} + c)
\end{pmatrix}.
\end{equation*}
Following \cite{CDB09} we can solve this system explicity in terms
of functions 
\begin{equation*}
\begin{aligned}
w^- (x; \kappa) &= e^{\sqrt{\kappa}x} \Big(a_0 + a_1 \tanh x + a_2 \tanh^2 x + \tanh^3 x \Big) \\
w^+ (x; \kappa) &= e^{-\sqrt{\kappa}x} \Big(-a_0 + a_1 \tanh x - a_2 \tanh^2 x + \tanh^3 x \Big),
\end{aligned}
\end{equation*}
where 
\begin{equation*}
a_0 = \frac{\kappa}{15} (4 - \kappa); \quad a_1 = \frac{1}{5} (2 \kappa - 3); \quad a_2 = - \sqrt{\kappa},
\end{equation*}
and the values of $\kappa$ will be specified below. 

We can now construct a basis for solutions decaying as $x \to -\infty$ as 
\begin{equation*}
\mathbf{p}_3^- (x; \lambda) = 
\begin{pmatrix}
w^- (x; -\lambda + 4) \\
w^- (x; -\lambda + 4)
\end{pmatrix};
\quad
\mathbf{p}_4^- (x; \lambda) = 
\begin{pmatrix}
- w^- (x; -\lambda + 4 + 2c) \\
w^- (x; -\lambda + 4 + 2c)
\end{pmatrix},
\end{equation*}
and a basis for solutions decaying as $x \to +\infty$ as 
\begin{equation*}
\mathbf{p}_1^+ (x; \lambda) = 
\begin{pmatrix}
w^+ (x; -\lambda + 4) \\
w^+ (x; -\lambda + 4)
\end{pmatrix};
\quad
\mathbf{p}_2^+ (x; \lambda) = 
\begin{pmatrix}
- w^+ (x; -\lambda + 4 + 2c) \\
w^+ (x; -\lambda + 4 + 2c)
\end{pmatrix}.
\end{equation*}

These considerations allow us to construct 
\begin{equation*}
\begin{aligned}
X^- (x; \lambda) &= 
\begin{pmatrix}
w^- (x; -\lambda + 4) & - w^- (x; -\lambda + 4 + 2c) \\
w^- (x; -\lambda + 4) & + w^- (x; -\lambda + 4 + 2c)
\end{pmatrix}; \\
X^+ (x; \lambda) &= 
\begin{pmatrix}
w^+ (x; -\lambda + 4) & - w^+ (x; -\lambda + 4 + 2c) \\
w^+ (x; -\lambda + 4) & + w^+ (x; -\lambda + 4 + 2c)
\end{pmatrix},
\end{aligned}
\end{equation*}
with then $Y^- (x; \lambda) = X^-_x (x;\lambda)$ and 
$Y^+ (x; \lambda) = X^+_x (x; \lambda)$. 

In order to construct the target space, we write (\ref{eg2E})
as a first-order system by setting $\Phi_1 = \phi$, 
$\Phi_2 = \psi$, $\Phi_3 = \phi'$, and $\Phi_4 = \psi'$.
This allows us to write 
\begin{equation*}
\mathbf{\Phi}' = \mathbb{A} (x; \lambda) \mathbf{\Phi}; 
\quad
\mathbb{A} (x; \lambda) = 
\begin{pmatrix}
0 & 0 & 1 & 0 \\
0 & 0 & 0 & 1 \\
- f(x;\lambda) & -c & 0 & 0 \\
-c & -f(x;\lambda) & 0 & 0
\end{pmatrix},
\end{equation*}
where $f(x; \lambda) = \lambda - 4 - c + 12 \bar{u}$. We set
\begin{equation*}
\mathbb{A}_+ (\lambda) := \lim_{x \to + \infty} \mathbb{A} (x; \lambda)
= \begin{pmatrix}
0 & 0 & 1 & 0 \\
0 & 0 & 0 & 1 \\
- \lambda + 4 + c & -c & 0 & 0 \\
-c & -\lambda + 4 + c & 0 & 0
\end{pmatrix}.
\end{equation*}

If we follow our usual ordering scheme for indices then for $-2 < c < 0$ we 
have $\nu_1^+ = 4 + 2c$ and $\nu_2^+ = 4$, with corresponding 
eigenvectors $r_1^+ = {1 \choose -1}$ and $r_2^+ = {1 \choose 1}$. 
Accordingly, we have $\mu_1^+ (\lambda) = - \sqrt{-\lambda + 4}$,
$\mu_2^+ (\lambda) = -\sqrt{-\lambda + 4 + 2c}$, 
$\mu_3^+ (\lambda) = \sqrt{-\lambda + 4 + 2c}$, 
and $\mu_4^+ (\lambda) = \sqrt{-\lambda + 4}$. We 
conclude that a frame for $\ell^+_{\mathbf{R}} (\lambda)$ 
is $\mathbf{R}^+ = {R^+ \choose S^+ (\lambda)}$, where 
\begin{equation*}
R^+ = 
\begin{pmatrix}
1 & 1 \\
1 & -1
\end{pmatrix}; \quad
S^+ (\lambda) = 
\begin{pmatrix}
\mu_1^+ (\lambda) & \mu_2^+ (\lambda) \\
\mu_1^+ (\lambda) & -\mu_2^+ (\lambda)
\end{pmatrix}.
\end{equation*}

The resulting spectral curves are plotted in Figure
\ref{example2_figure} for $c = -1$. In this case, it 
is known that $H$ has exactly six eigenvalues: 
$-7$, $-5$, $-2$, $0$, $1$ and $3$ (the eigenvalues 
$1$ and $3$ are omitted from our window). We see that 
the three crossings along the line $\lambda = 0$ correspond
with the count of three negative eigenvalues.     

\begin{figure}[ht] 
\begin{center}\includegraphics[%
  width=12cm,
  height=8cm]{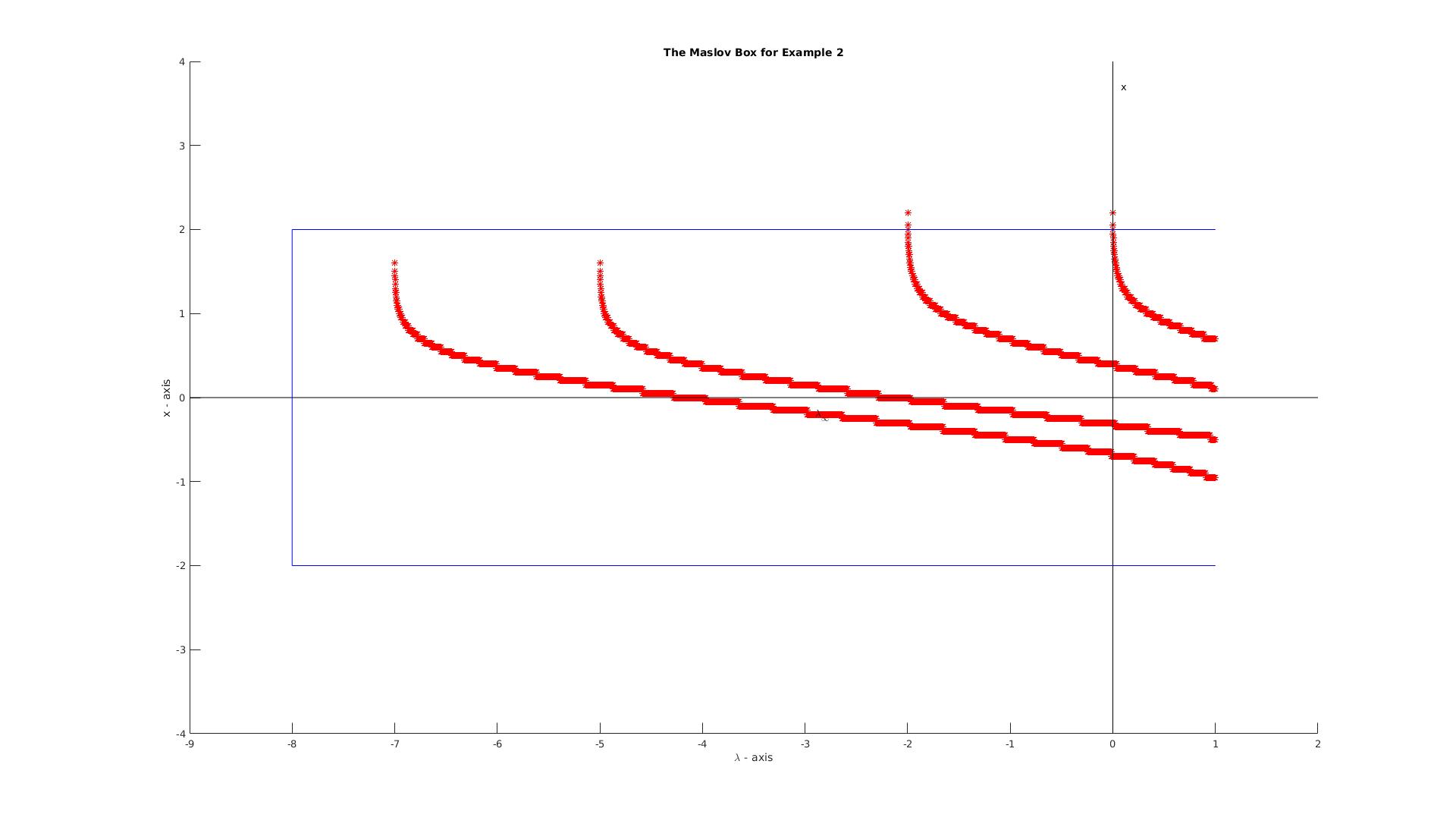}\end{center}
\caption{Figure for Example 2. \label{example2_figure}}
\end{figure}

\subsection{Example 3.} \label{example3_section}

Consider the Allen-Cahn system 
\begin{equation} \label{allencahn}
u_t  = u_{xx} - D_u F(u), 
\end{equation}
where 
\begin{equation*}
F(u_1, u_2) = u_1^2 u_2^2 + u_1^2 (1 - u_1 - u_2)^2 + u_2^2 (1 - u_1 - u_2)^2,
\end{equation*}
which is adapted from p. 39 of \cite{HK1}. In this setting, 
stationary solutions $\bar{u} (x)$ satisfying endstate
conditions 
\begin{equation*}
\lim_{x \to \pm \infty} \bar{u} (x) = u_{\pm}, 
\end{equation*}
for $u_- \ne u_+$ are called {\it transition waves}. A 
transition wave solution for (\ref{allencahn}) is depicted 
in Figure \ref{transitionfront}. In this case, we have
$u_1^- = 1$, $u_2^- = 0$, $u_1^+ = 0$, and $u_2^+ = 1$.

\begin{figure} \label{transitionfront}
\begin{center}\includegraphics[%
  width=12cm,
  height=8cm]{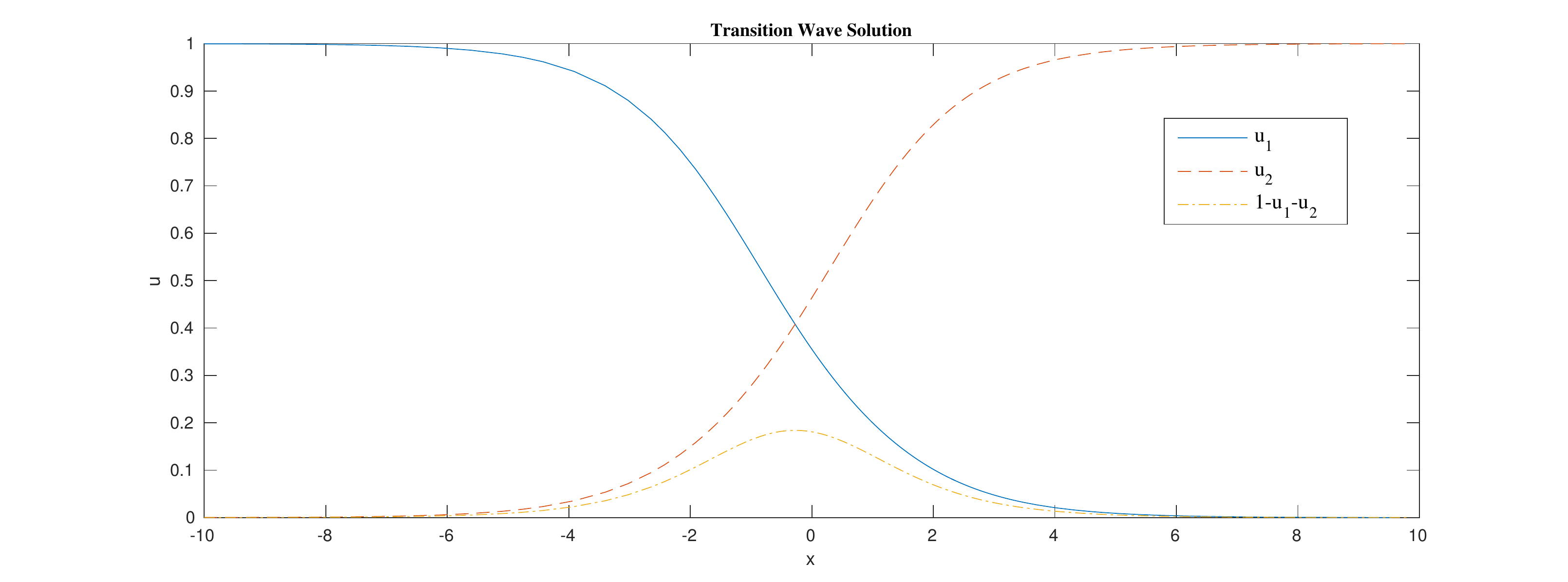}\end{center}
\caption{Transition front solution for a ternary Cahn-Hilliard system.}
\end{figure}

Upon linearization of (\ref{allencahn}) about $\bar{u} (x)$, we obtain 
the eigenvalue problem 
\begin{equation} \label{allencahnE}
- \phi'' + V (x) \phi = \lambda \phi, 
\end{equation}
where $V (x) := D_u^2 F(\bar{u})$ denotes the usual Hessian matrix. In 
this case, 
\begin{equation*}
V_- = 
\begin{pmatrix}
2 & 2 \\
2 & 4
\end{pmatrix};
\quad
V_+ = 
\begin{pmatrix}
4 & 2 \\
2 & 2
\end{pmatrix}.
\end{equation*}
Using our usual labeling scheme, we have $\nu_1^+ = 3 - \sqrt{5}$
and $\nu_2^+ = 3 + \sqrt{5}$, with respective eigenvectors 
\begin{equation*}
r_1^+ = 
\begin{pmatrix}
1 \\ -\frac{1+\sqrt{5}}{2} \\
\end{pmatrix} 
\quad
r_2^+ = 
\begin{pmatrix}
1 \\ -\frac{1-\sqrt{5}}{2}
\end{pmatrix}. 
\end{equation*} 
The corresponding values $\{\mu_j^+\}_{j=1}^4$ are 
$\mu_1^+ = - \sqrt{\nu_2^+ - \lambda}$, $\mu_2^+ = - \sqrt{\nu_1^+ -\lambda}$,
$\mu_3^+ = \sqrt{\nu_1^+ - \lambda}$, $\mu_4^+ = \sqrt{\nu_2^+ -\lambda}$.

For the target space $\ell^+_{\mathbf{R}}$ we use the frame 
\begin{equation*}
\mathbf{R}^+ (\lambda) = 
\begin{pmatrix}
R^+ \\ S^+ (\lambda)
\end{pmatrix}
= 
\begin{pmatrix}
r_2^+ & r_1^+ \\ 
\mu_1^+ r_2^+ & \mu_2^+ r_1^+
\end{pmatrix}.
\end{equation*}

For the evolving Lagrangian subspace $\ell^- (x; \lambda)$ we need a basis 
for the two-dimensional space of solutions that decay as $x \to -\infty$.
Generally, we construct this basis from the solutions 
\begin{equation*}
\mathbf{p}_{2+j}^- (x; \lambda) =  
e^{\mu_{2+j}^- (\lambda) x}
(\scripty{r}_{\,2+j}^{\,-} 
+ \mathbf{E}_{2+j}^-); \quad j = 1, 2, 
\end{equation*}
from Lemma \ref{ODElemma}, but computationally it is easier to 
note that for $\lambda = 0$, $\bar{u}_x$ is a solution of 
(\ref{main}) that decays as $x \to -\infty$. In \cite{HK1}
the authors check that $\bar{u}_x (x)$ decays at the slower
rate (i.e., the rate of $\mathbf{p}_3^-$), so we can take 
as our frame 
\begin{equation*}
\mathbf{X}^- (x; 0) = 
\begin{pmatrix}
p_4^- (x; 0) & \bar{u}_x (x) \\
p_4^{-\,\prime} (x; 0) & \bar{u}_{xx} (x)
\end{pmatrix},
\end{equation*} 
which we scale to 
\begin{equation*}
\mathbf{X}^- (x; 0) = 
\begin{pmatrix}
e^{-\mu_4^- (\lambda) x} p_4^- (x; 0) & e^{-\mu_3^- (\lambda) x} \bar{u}_x (x) \\
e^{-\mu_4^- (\lambda) x} p_4^{-\,\prime} (x; 0) & e^{-\mu_3^- (\lambda) x} \bar{u}_{xx} (x)
\end{pmatrix}.
\end{equation*}
The advantage of this is that $\bar{u} (x)$ is already known, and the 
faster-decaying solution $\mathbf{p}_4^- (x; 0)$ can be generated 
numerically in a straightforward way (see \cite{HK1}). 

In practice, we compute $\tilde{W} (x; 0)$ for $x$ running from $-10$
to $10$, and find that one of its eigenvalues remains confined to 
the semicircle with positive real part and that the other rotates 
monotonically clockwise, nearing $-1$ as $x$ approaches $10$. In 
this case, we know that $\lambda = 0$ is an eigenvalue, from which 
we can conclude that (at least) one of the eigenvalues will approach 
$-1$ as $x \to +\infty$. We view this as strong numerical evidence
(though certainly not numerical proof) that $\lambda = 0$ is a simple
eigenvalue of $H$, and that there are no negative eigenvalues of 
$H$.

\section*{Appendix}

In this short appendix, we construct the asymptotic Lagrangian path 
\begin{equation*}
\ell^-_{+\infty} (\lambda) = \lim_{x \to + \infty} \ell^- (x; \lambda),
\end{equation*}
and show that it is not generally continuous in $\lambda$. For a related 
discussion from a different point of view, we refer to Lemma 3.7 of 
\cite{AGJ}.

As a start, we recall that one choice of frame for $\ell^- (x;\lambda)$
is 
\begin{equation*}
\mathbf{X}^- (x; \lambda) 
= \begin{pmatrix}
\mathbf{p}^-_{n+1} (x; \lambda) & \mathbf{p}^-_{n+2} (x; \lambda) & \dots
& \mathbf{p}^-_{2n} (x; \lambda)
\end{pmatrix},
\end{equation*}
where we have from Lemma \ref{ODElemma}
\begin{equation*}
\mathbf{p}^-_{n+j} (x; \lambda) = e^{\mu_{n+j}^- (\lambda) x}
(\scripty{r}_{\,n+j}^{\,-} 
+ \mathbf{E}_{n+j}^-); \quad j = 1, 2, \dots, n. 
\end{equation*}
Each of the $\mathbf{p}^-_{n+j}$ can be expressed as a linear combination of 
the basis of solutions $\{\mathbf{p}_k^+\}_{k=1}^{2n}$, where we recall from 
Lemma \ref{ODElemma} that the solutions $\{\mathbf{p}_k^+\}_{k=1}^{n}$ decay
as $x \to +\infty$, while the solutions $\{\mathbf{p}_k^+\}_{k=n+1}^{2n}$
grow as $x \to +\infty$. I.e., for each $j = 1,2,\dots,n$, we can write 
\begin{equation*}
\mathbf{p}^-_{n+j} (x; \lambda)
= \sum_{k=1}^{2n} c_{j k} (\lambda) \mathbf{p}_k^+ (x;\lambda),
\end{equation*}
and so the collection of vector functions on the right-hand side
provides an alternative way to express the same frame $\mathbf{X}^- (x;\lambda)$.
We note that if $\lambda = \nu_{\min}$ then the modes
$\{\mathbf{p}_k^+ (x;\lambda)\}_{k=n+1}^{2n}$ cannot be obtained as 
$\lambda \to 0^-$ limits of solutions that grow as $x \to +\infty$, 
because these will coalesce with solutions obtained as $\lambda \to 0^-$ 
limits of solutions that decay as $x \to +\infty$. However, for such 
values of $\lambda$ we can still find $2n$ linearly independent solutions
of (\ref{first_order}), and the $\{\mathbf{p}_k^+ (x;\lambda)\}_{k=n+1}^{2n}$
correspond with the $n$ solutions not obtained as $\lambda \to 0^-$ 
limits of solutions that decay as $x \to +\infty$. For a direct approach 
toward defining such solutions, readers are referred to \cite{H}.

Fix $\lambda \in [-\lambda_{\infty}, \nu_{\min}]$, and suppose the fastest
growth mode $\mathbf{p}^+_{2n} (x; \lambda)$ appears in the expansion of 
at least one of the $\mathbf{p}^-_{n+j}$ (i.e., the coefficient associated with
this mode is non-zero). (There may be additional modes that grow at the same
rate $\mu_{2n}^+$, but they will have different, and linearly independent, 
associated eigenvectors $\scripty{r}_{\,n+j}^{\,-}$, allowing us to distinguish
them from $\mathbf{p}^+_{2n} (x; \lambda)$.) By taking appropriate linear 
combinations, we can identify a new frame for $\ell^- (x; \lambda)$ for which 
$\mathbf{p}_{2n}^+$ only appears in one column. If $\mathbf{p}^+_{2n} (x; \lambda)$
does not appear in the sum for any $\mathbf{p}^-_{n+j}$ we can start with 
$\mathbf{p}_{2n-1}^+$ and proceed similarly, continuing until we get to the first
mode that appears. Since the $\{\mathbf{p}^-_{n+j}\}_{j=1}^n$ form a basis
for an $n$-dimensional space, we will be able to distinguish $n$ modes in this
way. At the end of this process, we will have created a new frame for $\mathbf{X}^- (x; \lambda)$
with columns $\{\tilde{\mathbf{p}}^-_{j}\}_{j=1}^n$, where 
\begin{equation*}
\tilde{\mathbf{p}}_{j}^- (x; \lambda) = 
e^{\mu_{k(j)}^+ x} \Big(s_{k(j)}^+ + \tilde{\mathbf{E}}^+_{k(j)} (x; \lambda) \Big),
\end{equation*}  
for some appropriate map $j \mapsto k(j)$.
If the rate $\mu_{k (j)}^+$ is distinct as an eigenvalue of $\mathbb{A}_+ (\lambda)$
then we will have $s_{k(j)}^+ = \scripty{r}_{\,k(j)}^{\,+}$, but if $\mu_{k (j)}^+$
is not distinct then $s_{k(j)}^+$ will generally be a linear combination of 
eigenvectors of $\mathbb{A}_+ (\lambda)$ (and so, of course, still an eigenvector
of $\mathbb{A}_+ (\lambda)$). This process may also introduce an expansion coefficient
in front of $s_{k(j)}^+$, but this can be factored out in the specification of
the frame. 

As usual, we can view the exponential scalings $e^{\mu^+_{k(j)} x}$ as expansion coefficients, 
and take as our frame for $\ell^- (x;\lambda)$ the $2n \times n$ matrix with columns 
$s_{k(j)}^+ + \tilde{\mathbf{E}}^+_{k(j)} (x; \lambda)$. Taking now the limit 
$x \to \infty$ we see that we obtain the asymptotic frame 
\begin{equation}
\mathbf{X}^-_{+ \infty} (\lambda) 
= \begin{pmatrix}
s_{k(1)}^+ & s_{k(2)}^+ & \dots & s_{k(n)}^+
\end{pmatrix}.
\end{equation}
We can associate $\ell^-_{+\infty} (\lambda)$ as the Lagrangian subspace with 
this frame, verifying that this Lagrangian subspace is well-defined.

Last, we verify our comment that $\ell^-_{+\infty} (\lambda)$ is not generally 
continuous as a function of $\lambda$. To see this, we begin by noting that 
if $\lambda_0 \in [-\lambda_{\infty}, \nu_{\min}]$ is not an eigenvalue of 
$H$ then the leading modes selected in our process must all be growth modes, 
and we obtain $\mathbf{X}^-_{+ \infty} (\lambda_0) = \mathbf{R}^+ (\lambda_0)$, 
in agreement with Lemma 3.7 in \cite{AGJ}. Suppose, however, that 
$\lambda_0 \in [-\lambda_{\infty},\nu_{\min})$
is an eigenvalue of $H$, and for simplicity assume $\lambda_0$ has geometric
multiplicity 1. Away from essential spectrum, $\lambda_0$ will be isolated, and
so we know that any $\lambda$ sufficiently close to $\lambda_0$ will not 
be in the spectrum of $H$. We conclude that the frame for $\lambda_0$ will 
comprise $n-1$ of the eigenvectors $\{\scripty{r}_{\,n+j}^{\,+}\}_{j=1}^n$, along with 
one of the $\{\scripty{r}_{\,j}^{\,+}\}_{j=1}^n$. Since the exchanged vectors 
will lead to bases of different spaces, we can conclude that 
$\ell^-_{+\infty} (\lambda)$ is not continuous at $\lambda_0$. 

In order to clarify the discussion, we briefly consider the simple case $n=1$.
In this case, we have (for $\lambda < \nu_{\min}$) a single solution 
$\mathbf{p}_2^- (x;\lambda)$ that decays as $x \to - \infty$, and we can write 
\begin{equation*}
\mathbf{p}_2^- (x; \lambda) = c_{11} (\lambda) \mathbf{p}_1^+ (x; \lambda) 
+ c_{12} (\lambda) \mathbf{p}_2^+ (x; \lambda),
\end{equation*}
where $\mathbf{p}_1^+ (x; \lambda)$ decays as $x \to + \infty$ and 
$\mathbf{p}_2^+ (x; \lambda)$ grows as $x \to + \infty$. 
If $\lambda_0$ is not an eigenvalue of $H$ we must have $c_{12} (\lambda_0) \ne 0$,
and so 
\begin{equation*}
\begin{aligned}
\mathbf{p}_2^- (x; \lambda_0) &= 
c_{11} (\lambda_0) e^{\mu_1^+ (\lambda_0) x} (\scripty{r}_{\,1}^{\,+} 
+ \mathbf{E}_{1}^+ (x; \lambda_0))
+ c_{12} (\lambda_0) e^{\mu_2^+ (\lambda_0) x} (\scripty{r}_{\,2}^{\,+} 
+ \mathbf{E}_{2}^+ (x; \lambda_0)) \\
&= 
c_{12} (\lambda_0) e^{\mu_2^+ (\lambda_0) x} 
\Big(\scripty{r}_{\,2}^{\,+} + \mathbf{E}_{2}^+ (x; \lambda_0)
+ \frac{c_{11} (\lambda_0)}{c_{12} (\lambda_0)} e^{(\mu_1^+ (\lambda_0) - \mu_2^+ (\lambda_0)) x}
(\scripty{r}_{\,1}^{\,+} + \mathbf{E}_{1}^+ (x; \lambda_0)) \Big) \\
&=
c_{12} (\lambda_0) e^{\mu_2^+ (\lambda_0) x} 
(\scripty{r}_{\,2}^{\,+} + \tilde{\mathbf{E}}_{2}^+ (x; \lambda_0)), 
\end{aligned}
\end{equation*}
where $\tilde{\mathbf{E}}_{2}^+ (x; \lambda_0)) = \mathbf{O} ((1+|x|)^{-1})$.

We can view $\scripty{r}_{\,2}^{\,+} + \tilde{\mathbf{E}}_{2}^+ (x; \lambda_0)$ as 
a frame for $\ell^- (x; \lambda_0)$, and it immediately follows that as $x \to \infty$
the path of Lagrangian subspaces $\ell^- (x; \lambda_0)$ approaches the Lagrangian 
subspace with frame $\scripty{r}_{\,2}^{\,+}$ (denoted $\ell^-_{+ \infty} (\lambda_0)$
above). Moreover, since 
$\scripty{r}_{\,1}^{\,+}$ serves as a frame for $\ell^+_{\mathbf{R}} (\lambda_0)$ 
we can construct $\tilde{\mathcal{W}} (x; \lambda_0)$ from this pair. Taking the 
limit as $x \to \infty$ we see that 
\begin{equation*}
\tilde{\mathcal{W}}^+ (\lambda_0) := \lim_{x \to \infty} \tilde{\mathcal{W}} (x; \lambda_0)
= - \frac{r_1^+ + i \mu_2^+ (\lambda_0) r_1^+}{r_1^+ - i \mu_2^+ (\lambda_0) r_1^+}
\cdot \frac{r_2^+ - i \mu_1^+ (\lambda_0) r_2^+}{r_2^+ + i \mu_1^+ (\lambda_0) r_2^+}.
\end{equation*}  
By normalization, we can take both $r_1^+$ and $r_2^+$ to be $1$, and we must have 
$\mu_1^+ = - \mu_2^+$, so 
\begin{equation*}
\tilde{\mathcal{W}}^+ (\lambda_0) 
= - \frac{1 + i \mu_2^+ (\lambda_0)}{1 - i \mu_2^+ (\lambda_0)}
\cdot \frac{1 + i \mu_2^+ (\lambda_0)}{1 - i \mu_2^+ (\lambda_0)}
= - \frac{(1 + i \mu_2^+ (\lambda_0))^2}{(1 - i \mu_2^+ (\lambda_0))^2},
\end{equation*}
which can only be $-1$ if $\mu_2^+ (\lambda_0) = 0$ (a case ruled out 
in this calculation). 

On the other hand, if $\lambda_0 \in \sigma_{pt} (H)$ we will have $c_{12} (\lambda_0) = 0$
(and $c_{11} (\lambda_0) \ne 0$). In this case, the frame for $\ell^- (x; \lambda_0)$
will be $\scripty{r}_{\,1}^{\,+} + \tilde{\mathbf{E}}_{1}^+ (x; \lambda_0)$, and taking 
$x \to +\infty$ we see that $\ell^- (x; \lambda_0)$ will approach the Lagrangian subspace
with frame $\scripty{r}_{\,1}^{\,+}$. Recalling again that $\scripty{r}_{\,1}^{\,+}$ 
serves as a frame for $\ell^+_{\mathbf{R}} (\lambda_0)$ we see that        
\begin{equation*}
\begin{aligned}
\tilde{\mathcal{W}}^+ (\lambda_0) &:= \lim_{x \to \infty} \tilde{\mathcal{W}} (x; \lambda_0)
= - \frac{r_2^+ + i \mu_1^+ (\lambda_0) r_2^+}{r_2^+ - i \mu_1^+ (\lambda_0) r_2^+}
\cdot \frac{r_2^+ - i \mu_1^+ (\lambda_0) r_2^+}{r_2^+ + i \mu_1^+ (\lambda_0) r_2^+} \\
&= -1.
\end{aligned}
\end{equation*}  

For Example 1 in Section \ref{applications_section}, we have 
$\mu_1^+ (\lambda) = - \sqrt{1-\lambda}$ and 
$\mu_2^+ (\lambda) = + \sqrt{1 - \lambda}$. We know that in that example 
$\lambda_0 = 0$ is an eigenvalue, so we have $\tilde{\mathcal{W}}^+ (0) = -1$,
but for $\lambda \ne 0$, $|\lambda| < 1$, we have 
\begin{equation*}
\tilde{\mathcal{W}}^+ (\lambda) 
= - \frac{(1+i\sqrt{1-\lambda})^2}{(1-i\sqrt{1-\lambda})^2}. 
\end{equation*}
If we substitute $\lambda = 0$ into this relation, we obtain $+1$, 
and so we see that $\tilde{\mathcal{W}}^+ (\lambda)$ is not 
continous in $\lambda$ (at $\lambda = 0$ in this case). 



\begin{thebibliography}{XXXX}


\bibitem{A01} A. Abbondandolo,
{\em Morse Theory for Hamiltonian Systems.}
Chapman \& Hall/CRC Res. Notes Math. {\bf 425}, Chapman \& Hall/CRC, Boca Raton, FL, 2001.

\bibitem{AGJ} J. Alexander, R. Gardner, and C. Jones, 
{\it A topological invariant arising in the stability analysis of 
travelling waves}, J. reine angew. Math. {\bf 410} (1990) 167--212.


\bibitem{arnold67}  V. I. Arnold,
{\em Characteristic class entering in quantization conditions,}
Func. Anal. Appl. {\bf 1} (1967) 1 -- 14.

\bibitem{Arn85} V. I.  Arnold,
{\em The Sturm theorems and symplectic geometry,}
 Func. Anal. Appl. {\bf 19} (1985) 1--10.

\bibitem{At} F. V. Atkinson, {\em Discrete and Continuous Boundary Problems},
in the series {\em Mathematics in Science and Engineering} (vol. 8), 
Academic Press 1964.

\bibitem{B56} R.\ Bott,
{\em On the iteration of closed geodesics and the Sturm intersection theory,}
Comm. Pure Appl. Math. {\bf 9} (1956) 171 -- 206.

\bibitem{BF98} B.\ Booss-Bavnbek and K.\ Furutani, {\em The Maslov index: a functional analytical definition and the spectral flow formula,}
Tokyo J.\ Math.\ {\bf 21} (1998), 1--34.

\bibitem{BJ1995} A.\ Bose and C.\ K.\ R.\ T.\ Jones,
{\em Stability of the in-phase traveling wave solution in a pair 
of coupled nerve fibers,}
Indiana U. Math. J. {\bf 44} (1995) 189 -- 220.

\bibitem{BK} G.\ Berkolaiko and P.\ Kuchment, 
{\em Introduction to quantum graphs}, Mathematical 
Surveys and Monographs {\bf 186}, AMS 2013.

\bibitem{BM2013} M.\ Beck and S.\ Malham, 
{\em Computing the Maslov index for large systems}, 
Proc. Amer. Math. Soc. {\bf 143} (2015), no. 5, 2159–2173.

\bibitem{BeOr78} C. \ Bender and S. \ Orszag, {\em Advanced  Mathematical Methods for Scientists and Engineers.} McGraw-Hill, Sydney, 1978.

\bibitem{CDB06} F.\ Chardard, F.\ Dias and T.\ J.\ Bridges, {\em Fast computation of the Maslov index for hyperbolic linear systems with periodic coefficients.} J. Phys. A {\bf 39} (2006) 14545 -- 14557.

\bibitem{CDB09} F.\ Chardard, F.\ Dias and T.\ J.\ Bridges, {\em
 Computing the Maslov index of solitary waves. I. Hamiltonian systems on a four-dimensional phase space},
Phys. D {\bf 238} (2009) 1841 -- 1867.

\bibitem{CDB11} F.\ Chardard, F.\ Dias and T.\ J.\ Bridges, {\em Computing the Maslov index of solitary waves, Part 2: Phase space with dimension greater than four.} Phys. D {\bf 240} (2011) 1334 -- 1344.

\bibitem{CH2007} C-N Chen and X. Hu, {\em Maslov index for homoclinic 
orbits in Hamiltonian systems}, Ann. I. H. Poincar\'e -- AN {\bf 24} (2007) 589 -- 603.

\bibitem{Chardard2009} F.\ Chardard, 
{\em Stability of Solitary Waves}, Doctoral thesis, Centre de Mathematiques et de 
Leurs Applications, 2009. Advisor: T.\ J.\ Bridges. 

\bibitem{CJLS2014} G.\ Cox, C.\ K.\ R.\ T.\ Jones, Y.\ Latushkiun, and A.\ Sukhtayev, 
{\em The Morse and Maslov indices for multidimensional Schr\"odinger operators 
with matrix-valued potentials}, to appear in Transaction of the AMS. 

\bibitem{CZ84} C. Conley and E. Zehnder, {\em Morse-type index theory for flows and periodic solutions for Hamiltonian equations.} Comm. Pure Appl. Math. {\bf 37} (1984) 207 -- 253.

\bibitem{CLM}  S.\ Cappell, R.\ Lee and E.\ Miller, {\em On the Maslov index},
Comm.\ Pure Appl.\ Math.\ {\bf 47}  (1994), 121--186.

\bibitem{D76} J. J. Duistermaat, {\em On the Morse index in variational calculus.} Advances in Math. {\bf 21} (1976) 173 -- 195.

\bibitem{DJ11} J.\ Deng and C.\ Jones, {\em Multi-dimensional Morse Index Theorems and a symplectic view of elliptic boundary value problems,}
Trans. Amer. Math. Soc. {\bf 363} (2011) 1487 -- 1508.

\bibitem{DS} N.\ Dunford and J.\ T.\ Schwartz, {\em Linear Operators Part II: Spectral Theory}, John
Wiley \& Sons, Inc., 1988 reprint of 1963 edition. 

\bibitem{FJN03} R.\ Fabbri, R.\ Johnson and C.\ N\'u\~nez,
{\em Rotation number for non-autonomous linear Hamiltonian
systems I: Basic properties,}
Z. angew. Math. Phys. {\bf 54} (2003) 484 -- 502.

\bibitem{F} K.\ Furutani, {\em Fredholm-Lagrangian-Grassmannian and the Maslov index,} Journal of Geometry and Physics {\bf 51} (2004) 269 -- 331.

\bibitem{Ga93} R.\ A.\ Gardner,
{\em On the structure of the spectra of periodic travelling waves,}
 J. Math. Pures Appl. {\bf 72} (1993) 415 -- 439.

 \bibitem{G07} F.\ Gesztesy,
{\it Inverse spectral theory as influenced by Barry Simon}, In: {\em Spectral Theory and Mathematical Physics: a Festschrift in Honor of Barry Simon's 60th Birthday}, pp.\ 741 -- 820,
Proc. Sympos. Pure Math. {\bf 76}, Part 2, AMS, Providence, RI, 2007.

\bi {GLZ} F.\,Gesztesy, Y.\,Latushkin and K.\ Zumbrun,
{\it Derivatives of (modified) Fredholm determinants and stability of standing and traveling waves}, 
J. Math. Pures Appl. {\bf 90} (2008), 160--200.

\bibitem{GM2008} F. Gesztesy and M. Mitrea, {\it Generalized Robin boundary conditions, Robin-to-Dirichlet maps, and Krein-type
resolvent formulas for Schr\"odinger operators on bounded Lipschuitz domains},
in {\it Perspectives in Partial Differential Equations, Harmonic Analysis and Applications}, D. Mitrea and M. Mitrea
(eds.), Proceedings of Symposia in Pure Mathematics, American Mathematical Society, RI 2008.

\bibitem{GST96} F.\ Gesztesy, B.\ Simon and G.\ Teschl,
{\em Zeros of the Wronskian and renormalized oscillation theory}, Amer. J. Math. {\bf 118} (1996)  571 -- 594.

 \bibitem{GT09} F. Gesztesy and V. Tkachenko, {\em  A criterion for Hill operators to be spectral operators of scalar type.} J. Anal. Math. {\bf 107} (2009) 287 -- 353.

\bibitem{GW96} F. Gesztesy and R. Weikard, {\em Picard potentials and Hill's equation on a torus.} Acta Math. {\bf 176} (1996) 73 -- 107.

\bibitem{GZ} R. Gardner and K. Zumbrun, {\em The gap lemma and geometric criteria for instability of 
viscous shock profiles}, Comm. Pure Appl. Math. {\bf 51} (1998) 797-855.


\bibitem{H} P. Howard, {\em Asymptotic behavior near transition fronts for equations
of generalized Cahn-Hilliard form}, Comm. Math. Phys. {\bf 269} (2007) 765-808.

\bibitem{Henry} D.\ Henry, {\em Geometric theory of semilinear parabolic equations}, 
Lect. Notes Math. {\bf 840}, Springer-Verlag, Berlin-New York, 1981.

\bibitem{HK1} P. Howard and B. Kwon, {\it Spectral analysis for transition front
solutions to Cahn-Hilliard systems}, Discrete and Continuous Dynamical 
Systems A {\bf 32} (2012) 126-166.

\bibitem{HLS} P. Howard, Y. Latushkin, and A. Sukhtayev,
{\it The Maslov index for Lagrangian pairs on $\mathbb{R}^{2n}$}, 
Preprint 2016.

\bibitem{HS} P. Howard and A. Sukhtayev, 
{\em The Maslov and Morse indices for Schr\"odinger operators on $[0,1]$}, 
J. Differential Equations {\bf 260} (2016), no. 5, 4499-4549.

\bibitem{J88}  C.\ K.\ R.\ T.\ Jones,
{\em Instability of standing waves for nonlinear Schr\"odinger-type equations},
Ergodic Theory Dynam. Systems {\bf 8} (1988) 119 -- 138.

\bibitem{J88a}  C.\ K.\ R.\ T.\ Jones,
{\em An instability mechanism for radially symmetric standing
waves of a nonlinear Schr\"odinger equation},
J. Differential Equations {\bf 71} (1988) 34 -- 62.


\bibitem{JM2012} C.\ K.\ R.\ T.\ Jones and R.\ Marangell,
{\em The spectrum of travelling wave solutions to the Sine-Gordon
equation}, Discrete and Cont. Dyn. Sys. {\bf 5} (2012) 925 -- 937.

\bibitem{JS99} D.\ W.\ Jordan and P.\ Smith, {\it Nonlinear Ordinary Differential Equations: An Introduction to Dynamical Systems.} Oxford App. and Engin. Math., Oxford, 1999.

\bibitem{K97} Y. Karpeshina,
{\em Perturbation Theory for the Schr\"odinger Operator with a Periodic Potential.}
Lect. Notes Math. {\bf 1663},  Springer-Verlag, Berlin, 1997.

\bibitem{K} A.\ Krall, {\it Hilbert Space, Boundary Value Problems and Orthogonal Polynomials.
Operator Theory: Advances and Applications}, {\bf 133}, Birkhauser Verlag, Basel, 2002.

\bibitem{Kato} T.\ Kato, {\it Perturbation Theory for Linear Operators}, Springer, Berlin, 1980.

\bibitem{Keener} J. P. Keener, {\it Principles of Applied Mathematics: Transformation and 
Approximation}, 2nd Ed., Westview 2000. 

\bibitem{KP} T. Kapitula and K. Promislow, {\it Spectral and dynamical stability
of nonlinear waves}, Springer 2013.

\bibitem{Kuchment2004} P. Kuchment, {\it Quantum graphs: I. Some basic structures}, Waves in random
media {\bf 14}.

\bibitem{LS1} Y.\ Latushkin and A.\ Sukhtayev, {\it The Evans function and the Weyl-Titchmarsh function}, in
{\it Special issue on stability of travelling waves},  Disc. Cont. Dynam. Syst. Ser. S 5 (2012), no. 5, 939 - 970.


\bibitem{MW} W.\ Magnus and S.\ Winkler, {\em Hill's Equation},  Dover, New York, 1979.

\bi{Maslov1965a} V.\ P.\ Maslov, {\it Theory of perturbations and 
asymptotic methods}, Izdat. Moskov. Gos. Univ. Moscow, 1965. French 
tranlation Dunod, Paris, 1972. 


\bibitem{M63} J.\ Milnor, {\em Morse Theory},  Annals of Math.\ Stud. {\bf 51}, Princeton Univ. Press, Princeton, N.J., 1963.

\bibitem{O90} V.\ Yu.\ Ovsienko,  {\em Selfadjoint differential operators and curves on a Lagrangian Grassmannian that are subordinate to a loop,}  Math. Notes {\bf 47} (1990)  270 -- 275.

\bibitem{P96} J.\ Phillips, Selfadjoint Fredholm operators and spectral flow, {\em Canad.\ Math.\ Bull.} {\bf 39} (1996), 460--467.

\bibitem{RS78} M.\ Reed and B.\ Simon, {\it Methods of Modern Mathematical
Physics. IV: Analysis of Operators},  Academic Press, New York, 1978.
	
\bibitem{rs93} J.\ Robbin and D.\ Salamon,
{\em The Maslov index for paths}, Topology {\bf 32} (1993) 827 -- 844.

\bibitem{RoSa95} J.\ Robbin and D.\ Salamon, {\it  The spectral flow and the Maslov
index},  Bull. London Math. Soc. {\bf 27} (1995)  1--33.

\bibitem{SS08} B.\ Sandstede and A.\ Scheel, \emph{Relative Morse indices, Fredholm indices, and group velocities},  Discrete Contin. Dyn. Syst.  \textbf{20}  (2008)  139 -- 158.

\bi{W05} J. Weidman, {\em Spectral theory of Sturm-Liouville operators. Approximation by regular problems.} In:
{\em Sturm-Liouville Theory: Past and Present,} pp. 75--98,
W.\ O.\ Amrein, A.\ M.\ Hinz and D.\ B.\ Pearson, edts, Birkh\"auser, 2005.

\bibitem{ZH} K. Zumbrun and P. Howard, {\it Pointwise semigroup methods and stability of viscous shock waves},
Indiana U. Math. J. {\bf 47} (1998) 741-871. See also the errata for this paper: Indiana U. Math. J. {\bf 51}
(2002) 1017--1021. 

\end{thebibliography}
\end{document}